\pdfoutput=1 
\documentclass[11pt]{amsart}

\usepackage{epsfig,overpic,comment}
\usepackage[usenames,dvipsnames,svgnames,table]{xcolor}
\usepackage[pagebackref,linktocpage=true,colorlinks=true,linkcolor=Blue,citecolor=BrickRed,urlcolor=RoyalBlue]{hyperref}
\usepackage[msc-links,abbrev]{amsrefs} 
\usepackage{amsmath,amsthm,amssymb,marginnote}
\usepackage{enumerate}

\newtheorem{thm}{Theorem}[section]
\newtheorem{cor}[thm]{Corollary}
\newtheorem{lem}[thm]{Lemma}
\newtheorem{prop}[thm]{Proposition}

\theoremstyle{definition}
\newtheorem{note}[thm]{Note}

\newcommand{\be}{\begin{equation}}
\newcommand{\ee}{\end{equation}}
\newcommand{\ol}{\overline}
\newcommand{\R}{\mathbf{R}}

\newcommand{\mc}[1]{\mathcal{#1}}
\newcommand{\dt}{\frac d {dt}}

\newcommand{\C}{\mathcal{C}}

\renewcommand{\epsilon}{\varepsilon}
\renewcommand{\S}{\mathbf{S}}
\renewcommand{\tilde}{\widetilde}
\renewcommand{\P}{\mathcal{P}}

\DeclareMathOperator{\inte}{int}

\DeclareMathOperator{\conv}{conv}
\DeclareMathOperator{\dist}{dist}

\title{Shortest Closed Curve to Inspect a Sphere}
\author{Mohammad Ghomi}
\address{School of Mathematics, Georgia Institute of Technology,
Atlanta, GA 30332}
\email{ghomi@math.gatech.edu}
\urladdr{www.math.gatech.edu/~ghomi}

\author{James Wenk}
\address{School of Mathematics, Georgia Institute of Technology,
Atlanta, GA 30332}
\email{jwenk3@math.gatech.edu}
\urladdr{www.math.gatech.edu/~jwenk3}

\date{\today \,(Last Typeset)}
\thanks{Research of M.G. was supported in part by NSF Grant DMS--1711400 and a Simons fellowship.}

\begin{document}

\vspace*{-1in}
\maketitle

\vspace*{-0.25in}

\begin{abstract}
We show that in Euclidean 3-space any closed curve $\gamma$  which lies outside the unit sphere  and contains the sphere within its convex hull has length $\geq4\pi$. Equality holds  only when $\gamma$ is composed of $4$ semicircles  of  length  $\pi$, arranged in the shape of a baseball seam, as conjectured by  
V. A. Zalgaller in 1996. 
\end{abstract}


\section{Introduction}
What is the shortest closed orbit  a satellite may take to inspect  the entire surface of a round asteroid?
This is a well-known  optimization problem \cite{zalgaller:1996,orourkeMO,hiriart:2008,ghomi:lwr,cfg,finch&wetzel} in classical differential geometry and convexity theory, which may be precisely formulated  as follows.
A curve $\gamma$ in Euclidean space $\R^3$ \emph{inspects} a sphere $S$ provided that it lies outside $S$ and each point $p$ of $S$ can be ``seen" by some point $q$ of $\gamma$, i.e., the line segment $pq$ intersects $S$ only at $p$. It is easily shown that the latter condition holds if and only if $S$ lies in the convex hull of $\gamma$. The supremum of the radii of the spheres which are contained in the convex hull of $\gamma$ and are disjoint from $\gamma$  is called the \emph{inradius} of $\gamma$. Thus  we seek the shortest closed curve with a given inradius. The answer is as follows:

\begin{thm}\label{thm:main}
Let $\gamma\colon[a,b]\to\R^3$ be a closed rectifiable curve of length $L$ and inradius $r$. Then
\begin{equation}\label{eq:main}
L\geq 4\pi r.
\end{equation}
Equality holds only if, up to a reparameterization, $\gamma$ is  simple, $\C^{1,1}$, lies on a sphere of radius $\sqrt2 \,r$, and traces consecutively $4$ semicircles of length $\pi r$.
\end{thm}

\begin{figure}[h]
\begin{overpic}[height=1.25in]{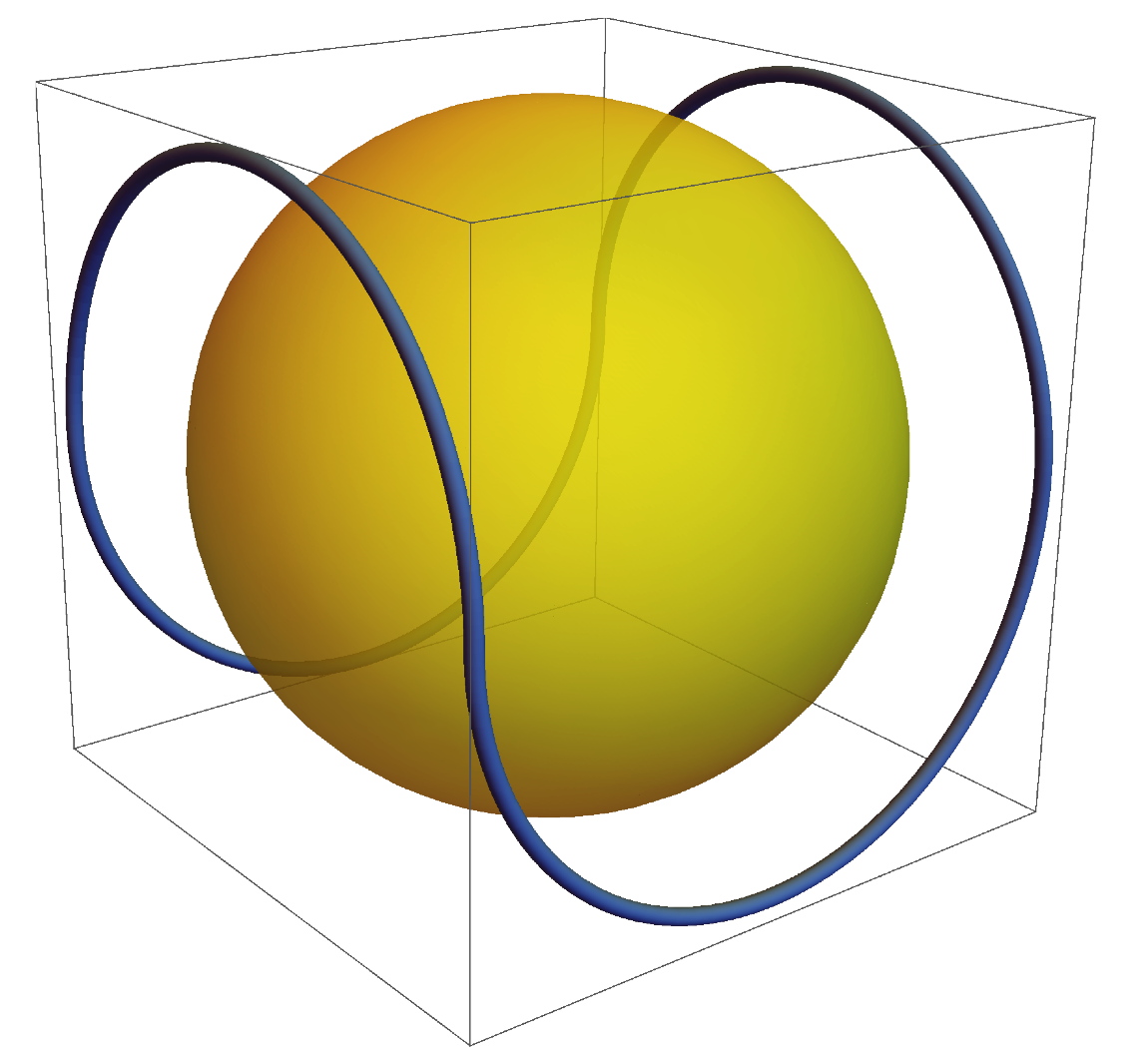}
\end{overpic}
\caption{}\label{fig:baseball}
\end{figure}

It follows that the image of the minimal  curve is unique up to a rigid motion, and
resembles the shape of a baseball seam as shown in Figure \ref{fig:baseball}, which settles a conjecture of Viktor Zalgaller made in 1996 \cite{zalgaller:1996}.  The previous best estimate was  $L\geq 6\sqrt{3}\,r$  obtained in 2018 \cite{ghomi:lwr}.  Here we use some notions from \cite{ghomi:lwr} together with other techniques from integral geometry (Crofton type formulas), geometric knot theory (unfoldings of space curves), and geometric measure theory (tangent cones, sets of positive reach) to establish the above theorem. We also derive several formulas (Sections \ref{sec:integral}, \ref{sec:horizon}, and \ref{sec:geq-sqrt2}) for efficiency of curves, which may be verified with a software package that we have provided \cite{ghomi-wenk:Mathematica2}.

Our main approach for proving Theorem \ref{thm:main} is as follows. Since \eqref{eq:main} is invariant under rescaling and rigid motions,  we may assume that $r=1$ and $\gamma$ inspects the unit sphere $\S^2$, in which case we say simply that $\gamma$ is an \emph{inspection curve}. Then we define the \emph{horizon} of  $\gamma$ (Section \ref{sec:integral}) as the measure in $\S^2$ counted with multiplicity of the set of points $p\in\S^2$ where the tangent plane $T_p\S^2$ intersects $\gamma$:
$$
H(\gamma):=\int_{p\in\S^2} \#\gamma^{-1}(T_p\S^2)\, dp.
$$
Since $\gamma$ is closed, $T_p\S^2$ intersects $\gamma$ at least twice for almost every $p\in\S^2$. Thus $H(\gamma)\geq 8\pi$. Next we define the \emph{(inspection) efficiency} of  $\gamma$ as 
\be\label{eq:E}
E(\gamma):=\frac{H(\gamma)}{L(\gamma)}.
\ee
So to establish \eqref{eq:main} it suffices to show that $E(\gamma)\leq 2$. Now note that, since $H$ is additive, for any partition of  $\gamma$ into subsets $\gamma_i$,  $i\in I$, 
\be\label{eq:EHL}
E(\gamma)= \sum_i \frac{H(\gamma_i)}{L(\gamma)}=\sum_i \frac{L(\gamma_i)}{L(\gamma)} E(\gamma_i)\leq \sup_i E(\gamma_i) .
\ee
So the desired upper bound  for $E(\gamma)$ may be established through a partitioning of $\gamma$ into subsets $\gamma_i$ with $E(\gamma_i)\leq 2$.  

To find the desired partition, we may start by assuming that $\gamma$ in Theorem \ref{thm:main} has minimal length among all  (closed) inspection curves, and is parameterized with constant speed (Section \ref{sec:prelim}). Then we apply an ``unfolding" procedure \cite{cks2002} to transform $\gamma$ into a planar curve $\tilde\gamma$ with the same arclength and  \emph{height}, i.e., radial distance function from the origin $o$ of $\R^3$ (Section \ref{sec:unfolding}). It follows that $E(\gamma)=E(\tilde\gamma)$. Furthermore, the minimality of $\gamma$ will ensure that $\tilde\gamma$ is ``locally convex with respect to o" \cite{ghomi2019}. Consequently $\tilde\gamma$ may  be partitioned into a collection of curves $\tilde\gamma_i$ we call
spirals (Section \ref{sec:decomposition}). A \emph{spiral} is a planar curve which lies outside the unit circle $\S^1$,  is locally convex with respect to $o$, has monotone height, and is orthogonal to the position vector of its closest boundary point to $o$. We will show that  the efficiency of any spiral is at most $2$ by polygonal approximations  and a variational argument (Section \ref{sec:inequality}), which establish \eqref{eq:main}. 

The rest of the paper will be devoted to characterizing minimal inspection curves.  First we note that equality holds in \eqref{eq:main} only when $E(\gamma)=2$, which forces all the  spirals  $\tilde\gamma_i$ to have efficiency $2$ as well. Then we show that a spiral has efficiency $2$ only when it has constant height $\sqrt2$  (Section \ref{sec:leq-sqrt2}) by refining the variational procedure used earlier to establish \eqref{eq:main}. 
Once we know that  any minimal inspection curve $\gamma$ has constant height $\sqrt{2}$, we proceed to the final stages of the characterization (Section \ref{sec:proof}). 
The simplicity of $\gamma$  follows from a Crofton type formula of Blaschke-Santalo. Then a characterization of $\C^{1,1}$ submanifolds in terms of their tangent cones \cite{ghomi-howard2014} and reach ensures the regularity of $\gamma$. Finally we show that $\gamma$ is composed of $4$ semicircles by constructing a ``nested partition" of $\gamma$ \cite{ghomi:rosenberg}, which completes the proof of Theorem \ref{thm:main}. The last technique goes back to the proofs of the classical $4$-vertex theorem due to Kneser and Bose \cite{hkneser,bose}, which has been  developed further by Umehara and Thorbergsson \cite{umehara2, thorbergsson&umehara}.
 
 The question we study in this work belongs to a circle of long standing optimization problems for the length of a curve in Euclidean space subject to various constraints on its convex hull, including bounds on volume, surface area, width, and inradius \cite{ghomi:lwr,zalgaller:1994,zalgaller:1996,zalgaller:2003,finch2019translation,Schoenberg:convexhull} \cite[A28, A30]{cfg}. Most of these problems remain open; see \cite{ghomi:lwr,tilli2010} for more background and references. We should also note that these problems may be posed both for closed and open curves. In the latter case, there are connections to the ``lost in a forest problem" of Bellman \cite{bellman:1956}, or its dual version, Moser's ``worm problem" \cite{brass-moser-pach2005, fassler-orponen2018, norwood-poole2003}, which are well-known in computational geometry.

 \section{Preliminaries: Minimal Inspection Curves}\label{sec:prelim}
 The central objects of study in this work are rectifiable curves, which become Lipschitz mappings after reparameterization with constant speed. We begin by recording  basic facts we need
 in this regard; see \cite{cesari1958},  \cite[Chap. 2]{bbi:book}, or \cite[chap. 4]{ambrosio-tilli2004} for more background.  
  Here $\R^n$ is the $n$-dimensional Euclidean space with origin $o$, inner product $\langle\cdot,\cdot\rangle$, and norm $|\cdot|:=\langle\cdot,\cdot\rangle^\frac{1}{2}$; $\S^{n-1}$, $B^n$ denote respectively the unit sphere and closed unit ball in $\R^n$.  The \emph{interior}, \emph{closure}, and \emph{boundary} of any set $X\subset\R^n$ is denoted by $\inte(X)$, $\ol X$, and $\partial X$ respectively. 
  A \emph{curve} is a continuous map $\gamma\colon[a,b]\to\R^n$, where $[a,b]\subset\R$ is an interval with $a<b$. We will also use $\gamma$ to refer to its image, $\gamma([a,b])$. We say that $\gamma$ is \emph{closed} if $\gamma(a)=\gamma(b)$. A closed curve $\gamma$ is \emph{simple} if it is one-to-one on $[a,b)$, and is  $\C^1$ provided that it is continuously differentiable with $\gamma'_+(a)=\gamma'_-(b)$; $\gamma$ is $\C^{1,1}$ if $\gamma'$ is Lipschitz. 
  The \emph{length} of $\gamma$ is  
$
L(\gamma):=\sup\sum_{i=1}^n \big|\gamma(t_i)-\gamma(t_{i-1})\big|,
$
where the supremum is taken over all partitions $a:=t_0\leq t\leq\dots\leq t_n:=b$ of $[a,b]$; $\gamma$ is \emph{rectifiable} if $L(\gamma)<\infty$, and has \emph{constant speed} $C$ if 
$
L(\gamma|_{[t,s]})=C\,|t-s|
$
for all  $t<s\in [a,b]$. 
A curve $\widehat\gamma\colon [a,b]\to\R^n$ is a \emph{reparameterization} of $\gamma$ if there is a nondecreasing continuous map $\phi\colon [a,b]\to[a,b]$ with $\gamma=\widehat\gamma\circ\phi$.  It is  well-known \cite[Prop. 2.5.9]{bbi:book} that any rectifiable curve  admits a  reparameterization with constant speed.
 If $\gamma\colon[a,b]\to\R^n$ has constant speed $C$, then for all $t<s\in[a,b]$,
\be\label{eq:const-speed}
|\gamma(t)-\gamma(s)|\leq L\big(\gamma\big|_{[t,s]}\big)= C|t-s|.
\ee
So
$\gamma$ is \emph{$C$-Lipschitz}, and therefore differentiable almost everywhere  by Rademacher's theorem. Then 
$
L(\gamma)=\int_a^b |\gamma'(t)|\,dt
$
\cite[Thm. 2.7.6]{bbi:book}.
Furthermore, \eqref{eq:const-speed} implies  that $|\gamma'|\leq C$ at all differentiable points of $\gamma$. On the other hand,  $\int_a^b |\gamma'(t)|\,dt/(b-a)=L(\gamma)/(b-a)=C$. Thus $|\gamma'|=C$ almost everywhere. So we record:

\begin{lem}\label{lem:basic}
Let $\gamma\colon[a,b]\to\R^n$ be a rectifiable curve. Then 
$\gamma$ has constant speed  if and only if $|\gamma'|=L(\gamma)/(b-a)$ almost everywhere.
\end{lem}

\noindent In particular, when $L(\gamma)>0$,  we may assume after reparameterization  with constant speed that 
$|\gamma'|\neq 0$ almost everywhere.
Let $\C^0([a,b],\R^n)$ denote the space of curves $\gamma\colon[a,b]\to\R^n$ with the supremum norm or \emph{uniform metric} \cite[p. 47]{bbi:book} given by
\be\label{eq:sup-norm}
\dist(\gamma_1,\gamma_2):=\sup_{t\in[a,b]}|\gamma_1(t)-\gamma_2(t)|.
\ee
The functional $L\colon \C^0([a,b],\R^n)\to \R$ is lower semi-continuous \cite[Prop. 2.3.4(iv)]{bbi:book}.
The \emph{convex hull} of a set $X\subset\R^n$, $\conv(X)$, is the intersection of all closed half-spaces containing $X$. We say $\gamma\colon[a,b]\to\R^3$ is an \emph{inspection curve} if it is closed and $\S^2\subset\conv(\gamma)$.  Lemmas  \ref{lem:basic} together with Arzela-Ascoli theorem \cite[Thm. 2.5.14]{bbi:book} and semicontinuity of $L$ yield \cite[Prop. 2.3]{ghomi-wenk-arXiv2020}:

\begin{prop}\label{prop:minimizer}
There exists an inspection curve of minimum length.
\end{prop}

Any curve given by the above proposition will be called a \emph{minimal} inspection curve. Let $\angle(v,w):=\cos^{-1}(\langle v, w\rangle/(|v||w|))$ denote the \emph{angle} between $v$, $w\in\R^n\setminus\{o\}$. For any rectifiable  curve $\gamma\colon[a,b]\to\R^n\setminus\{o\}$, with $L(\gamma)>0$, we set
$$
\alpha(t):=\angle \big(\gamma(t),\gamma'(t)\big).
$$
By Lemma \ref{lem:basic}, if $\gamma$ has constant speed, then  $|\gamma'|\neq 0$ almost everywhere. So
$\alpha$ is well-defined for almost every $t\in[a,b]$.  The \emph{tangent cone} $T_t\gamma$ of $\gamma$ at  $t\in[a,b]$ is the collection of all rays emanating from $\gamma(t)$ which are limits of a sequence of secant lines emanating from $\gamma(t)$ and passing through points $\gamma(s_i)$ as $s_i$ converge to $t$. If $\gamma$ is closed and $t=a$ or $b$, then we set $T_t\gamma:=T_a\gamma\cup T_b\gamma$. See \cite[Sec. 2]{ghomi-howard2014} for basic facts on tangent cones.
If $T_t\gamma$ is a line, then we call it the \emph{tangent line} of $\gamma$ at $t$. When $\gamma$ is differentiable at $t$ and $|\gamma'(t)|\neq 0$, $T_{t}\gamma$ is the line through $\gamma(t)$ spanned by $\gamma'(t)$. The following lemma generalizes an earlier observation   \cite[Lem. 7.4]{ghomi:lwr} for  polygonal curves. 

\begin{lem}\label{lem:alpha}
Let $\gamma\colon [a,b]\to \R^3$ be a constant speed minimal inspection curve. Then  tangent lines of $\gamma$ avoid   $\inte(B^3)$. In particular, for almost every $t\in[a,b]$,
\be\label{eq:sine-alpha}
\alpha(t)\geq\sin^{-1}\big( 1/|\gamma(t)|\big).
\ee
\end{lem}
\begin{proof}
Let  $T$ be a tangent line of $\gamma$ at  $t\in[a,b]$. Suppose that $T$ intersects $\inte(B^3)$. Set $X:=\conv(\{\gamma(t)\}\cup B^3)$.
 Then $T$ intersects  $\inte(X)$. So there is a open interval $U\subset [a,b]$ of the form $(t, s)$ or $(s,t)$ such that $\gamma(U)\subset \inte(X)$, and $\gamma(s)$ lies on $\partial X$.  
 Let $\ol U$ be the closure of $U$.
 Replacing $\gamma(\ol U)$ with a  line segment connecting $\gamma(t)$ and $\gamma(s)$  yields a closed curve $\beta$ with $L(\beta)<L(\gamma)$. But $\conv(\beta)=\conv(\gamma)$, since $\gamma(U)\subset\inte(\conv(X))\subset\inte(\conv(\gamma))$.  In particular $\S^2\subset\conv(\beta)$. So $\beta$ is an inspection curve shorter than $\gamma$, which is a contradiction. Now \eqref{eq:sine-alpha} follows from basic trigonometry.
\end{proof}

\section{The Integral Formula for Efficiency}\label{sec:integral}
As mentioned in the introduction, the \emph{efficiency} of any rectifiable curve $\gamma\colon[a,b]\to\R^3$ with $|\gamma|\geq 1$ is defined as
$$
E(\gamma):=\frac{H(\gamma)}{L(\gamma)},\quad\quad\text{where}\quad\quad H(\gamma):=\int_{p\in\S^2} \#\gamma^{-1}(T_p\S^2)\, dp.
$$
Recall that $H(\gamma)$ is  called the \emph{horizon} of $\gamma$. When $\gamma$ is an inspection curve it follows from Caratheodory's convex hull theorem that 
$
\#\gamma^{-1}(T_p\S^2)\geq 2
$ 
for almost every $p\in \S^2$ \cite[Lemma 7.1]{ghomi:lwr}. Thus $H(\gamma)\geq 8\pi$. So to prove  \eqref{eq:main} it suffices to show that $E(\gamma)\leq 2$. To this end
we will use the area formula  to compute $E(\gamma)$. This generalizes previous work \cite[Sec. 7.2]{ghomi:lwr} where the following proposition had been established for $\C^{1,1}$ curves. Recall that $\alpha:=\angle(\gamma,\gamma')$.

\begin{prop}\label{prop:H}
Let $\gamma\colon[a,b]\to\R^3$ be a constant speed curve with $|\gamma|\geq 1$. Then
\begin{equation}\label{eq:H2}
E(\gamma)= \frac{1}{b-a}\int_a^b\int_0^{2\pi} \frac{1}{|\gamma|^2}\left| \sqrt{|\gamma|^2-1}\sin\big(\alpha\big)\cos(\theta)+\cos\big(\alpha\big) \right|\,d\theta dt.
\end{equation}
\end{prop}
\begin{proof}
Let $\ol\gamma:={\gamma}/{|\gamma|}$. Since $\gamma$ is Lipschitz, $\ol\gamma$ is Lipschitz as well. So there exists a point $x\in \S^2\setminus \ol\gamma$.
Let $e_1$ be a $\C^1$ unit tangent vector field on $\S^2\setminus\{x\}$, and $e_2(p):=p\times u(p)$. Then $(e_1, e_2)$ is a Lipschitz  frame on any compact subset of $\S^2\setminus\{x\}$. So if we set
$
e_1(t):=e_1(\ol\gamma(t)), \; e_2(t):=e_2(\ol\gamma(t)),
$
then $t\mapsto (\ol\gamma(t), e_1(t), e_2(t))$ is a Lipschitz  frame along $\gamma$. Set
$
\lambda:=1/|\gamma|,\; \rho:=\sqrt{1-\lambda^2},
$
and define $F\colon[a,b]\times[0,2\pi]\to\S^2$ by
$$
F(t,\theta)=\lambda(t)\ol\gamma(t)+\rho(t)\big( \cos(\theta) e_1(t)+\sin(\theta)e_2(t)\big).
$$
Then $\theta\mapsto F(t,\theta)$ parameterizes the \emph{horizon circle} $H(\gamma(t))$, i.e., the set of points in $\S^2$ generated by  the rays which emanate from $\gamma(t)$ and are tangent to $\S^2$. So, for all $p\in\S^2$, 
$
F^{-1}(p)=\gamma^{-1}(T_p \S^2).
$
Thus, since $F$ is Lipschitz, the area formula \cite[Thm 3.2.3]{federer:book} yields
$$
H(\gamma)=\int_{p\in\S^2}\#F^{-1}(p)\,dp=\int_a^b\int_0^{2\pi}JF(t, \theta)\,d\theta dt,
$$
where $JF:=|\partial F/\partial t\times \partial F/\partial\theta|$ is the Jacobian of $F$. Next, for every differentiable point $t\in[a,b]$ of $\gamma$ let
$
E_\gamma(t):= \int_0^{2\pi}JF(t,\theta)\,d\theta.
$
By the Lebesgue differentiation theorem, for almost every $t\in[a,b]$,
$
E_\gamma(t)=\lim_{\epsilon\to 0}\frac{1}{2\epsilon}\int_{t-\epsilon}^{t+\epsilon}E_\gamma(s)\,ds=
\lim_{\epsilon\to 0}\frac{1}{2\epsilon}H(\gamma|_{[t-\epsilon,t+\epsilon]}).
$
So  $E_\gamma(t)$ does not depend on the choice of the frame $(e_1,e_2)$. We claim that for almost every point $t_0\in[a,b]$ of $\gamma$ we may choose a frame   so that
\be\label{eq:JF}
JF(t_0,\theta)= \frac{1}{|\gamma(t_0)|^2}\left| \sqrt{|\gamma(t_0)|^2-1}\sin\big(\alpha(t_0)\big)\cos(\theta)+\cos\big(\alpha(t_0)\big) \right|\,|\gamma'(t_0)|.
\ee
This would complete the proof because $E(\gamma)=(\int_a^b E_\gamma(t)\,dt)/L(\gamma)$, and since the speed is constant  $|\gamma'|=L(\gamma)/(b-a)$.
To establish \eqref{eq:JF}  note that if $t_0$ is a differentiable point of $\gamma$, then it is a differentiable point of $\ol\gamma$ as well. There are  two cases to consider: either $\ol\gamma'(t_0)\neq 0$ or 
$\ol\gamma'(t_0)= 0$.
First suppose that $\ol\gamma'(t_0)\neq 0$. Let $C$ be the great circle in $\S^2$ which is tangent to $\ol\gamma$ at $\ol\gamma(t_0)$. Set $e_1(\ol\gamma(t_0)):=\ol\gamma'(t_0)/|\ol\gamma'(t_0)|$. We may extend $e_1$ smoothly to a unit tangent vector field in a neighborhood of $\ol\gamma(t_0)$ on $\S^2$ so that $e_1(p)$ is tangent to $C$ when $p\in C$. Let $v:=|\ol\gamma'(t_0)|$. Then $\ol\gamma'(t_0)=ve_1(t_0)$, $e_1'(t_0)=-v\ol\gamma(t_0)$, and $e_2'(t_0)=0$.
Using these rules one may compute \cite{ghomi-wenk-arXiv2020,ghomi-wenk:Mathematica2} that  at $t=t_0$,
$
JF=|\rho v \cos(\theta)-\lambda'|
$
which is equivalent to \eqref{eq:JF}. If $\ol\gamma'(t_0)=0$, then for any choice of frame, $e_1'(t_0)=e_2'(t_0)=0$ by the chain rule. Then a computation \cite{ghomi-wenk-arXiv2020,ghomi-wenk:Mathematica2} shows that  $JF=|\gamma'|/|\gamma|^2$,  which establishes \eqref{eq:JF} since $\alpha=0$ or $\pi$.
\end{proof}

If $\gamma\colon[a,b]\to\R^3$ has constant  speed $C$, then
$
\big|\, |\gamma(s)|-|\gamma(t)|\,\big|\leq |\gamma(t)-\gamma(s)|\leq  C|t-s|.
$
So the function $|\gamma|\colon[a,b]\to\R$, which we call the \emph{height} of $\gamma$, is Lipschitz. In particular, $|\gamma|$ is differentiable almost everywhere. Furthermore note that if $t$ is a differentiable point of both $\gamma$ and $|\gamma|$, then 
$|\gamma|'=\langle \gamma,\gamma'\rangle/|\gamma|$ at $t$.
Thus for almost every $t\in[a,b]$
\be\label{eq:alpha-gamma-prime0}
\alpha(t)=\cos^{-1}(|\gamma|'(t)/C).
\ee
This shows, via  Proposition \ref{prop:H}, that $E(\gamma)$ depends only on $|\gamma|$. Hence we conclude

\begin{cor}\label{cor:lengthandheight}
Let $\gamma_1$, $\gamma_2\colon[a,b]\to\R^3$ be  constant speed curves with $L(\gamma_1)=L(\gamma_2)$. Furthermore suppose that $|\gamma_1(t)|=|\gamma_2(t)|\geq 1$ for all $t\in[a,b]$. Then $E(\gamma_1)=E(\gamma_2)$.
\end{cor}

\section{Unfolding of Minimal Inspection Curves}\label{sec:unfolding}
Here we describe a natural ``unfolding" procedure \cite{cks2002} which transforms a space curve  into a planar one. This operation preserves the arclength and height of the curve, and thus preserves its efficiency due to the results of the last section. Furthermore we will show that the unfolding of any minimal inspection curve satisfies a certain convexity condition. Let $\gamma\colon[a,b]\to\R^3\setminus\{o\}$ be a rectifiable curve. We set $\ol\gamma:=\gamma/|\gamma|$, and let
$
\theta_\gamma(t):=L\big(\ol\gamma\big|_{[a,t]}\big)=\int_a^t|\ol\gamma'(t)|dt
$
denote the arclength function of $\ol\gamma$ ($\theta_\gamma$ measures the ``cone angle" \cite{cks2002} or ``vision angle" \cite{choe-gulliver1992, choe-gulliver2017} of $\gamma$ from the point of view of $o$).
The \emph{(cone) unfolding}   of $\gamma$ is the planar curve $\tilde\gamma\colon[a,b]\to\R^2$ given by
$
\tilde\gamma(t):=|\gamma(t)|e^{i\theta_\gamma(t)},
$
where $e^{i\theta_\gamma}=(\cos(\theta_\gamma), \sin(\theta_\gamma))$.
In other words, $\tilde\gamma$ is generated by the isometric immersion (or unrolling) into $\R^2$ of the conical surface generated by the line segments $o\gamma(t)$. Note that $|\tilde\gamma(t)|=|\gamma(t)|$. Assuming $\gamma$ is reparameterized with constant speed,
\be\label{eq:theta-prime}
\tilde\gamma'=(|\gamma|'+i|\gamma|\theta_\gamma')e^{i\theta_\gamma},\quad\quad\text{and}\quad\quad 
\theta_\gamma'=|\ol\gamma'|=\frac{1}{|\gamma|^2}\sqrt{|\gamma|^2|\gamma'|^2-\langle\gamma,\gamma'\rangle^2},
\ee
almost everywhere.
Thus it follows that, for almost all $t\in[a,b]$,
\begin{eqnarray*}
|\tilde\gamma'|^2= 
(|\gamma|' )^2+|\gamma|^2(\theta_\gamma')^2
=\frac{\langle\gamma,\gamma'\rangle^2}{|\gamma|^2}+\frac{1}{|\gamma|^2}\Big(|\gamma|^2|\gamma'|^2-\langle\gamma,\gamma'\rangle^2\Big)=|\gamma'|^2.
\end{eqnarray*}
So $\gamma$ and $\tilde\gamma$ have equal height and length. Hence, by Corollary \ref{cor:lengthandheight}, 

\begin{prop}\label{prop:unfolding}
Let $\gamma\colon[a,b]\to\R^3$ be a rectifiable curve with $|\gamma|\neq 0$, and $\tilde\gamma$ be the unfolding of $\gamma$. Then $E(\gamma)=E(\tilde\gamma)$.
\end{prop}

Next we develop some geometric properties of  $\tilde\gamma$. 

\begin{lem}\label{lem:1-1}
Let $\gamma\colon[a,b]\to\R^3$ be a minimal inspection curve with constant speed. Then $\tilde\gamma$ is locally one to one.
\end{lem}
\begin{proof}
It suffices to show that $\theta'_\gamma>0$ almost everywhere. The formula for $\theta'_\gamma$ in \eqref{eq:theta-prime}, via the Cauchy-Schwartz inequality, shows that $\theta'_\gamma\geq 0$, and $\theta'_\gamma=0$ only when $\gamma'$ vanishes, or else $\gamma$ and $\gamma'$ are parallel. But $\gamma'$ can vanish only on a set of measure zero, since $\gamma$ has constant speed. Furthermore  if $\gamma$ and $\gamma'$ are parallel, then  $\alpha=0$. But by Lemma \ref{lem:alpha}, $\alpha\neq 0$ almost everywhere, which completes the proof.
\end{proof}

A \emph{convex body} $K\subset\R^2$ is a compact convex set with interior points.
A planar curve $\gamma\colon[a,b]\to\R^2$ is \emph{locally convex} if it is locally one-to-one and each  $t\in [a,b]$ has an open neighborhood $U\subset[a,b]$ such that $\gamma(U)$ lies on the boundary of a convex body. A \emph{local supporting line} $\ell$ for $\gamma$ at $t$ is a line passing through $\gamma(t)$ with respect to which  $\gamma(U)$ lies on one side. If $\ell$ does not pass through $o$ and $\gamma(U)$ lies on the side of $\ell$ which contains $o$, then  $\ell$ lies \emph{above} $\gamma$. 
If $\gamma$ is locally convex and through each point of it there passes a local support line which lies above $\gamma$, then  
$\gamma$ is locally convex \emph{with respect to $o$}.

\begin{prop}\label{prop:loc-convex}
Let $\gamma\colon[a,b]\to\R^3$ be a minimal inspection curve with constant speed. Then $\tilde\gamma$ is locally convex with respect to $o$.
\end{prop}
\begin{proof}
By Lemma \ref{lem:1-1}, every $t\in[a,b]$ has a neighborhood $U\subset[a,b]$ such that $\tilde\gamma$ is one-to-one on 
$U$. Assuming that $U$ is small, $\tilde\gamma(U)$ will be star-shaped with respect to $o$,
i.e., for every $s\in U$ the line passing through $o$ and $\tilde\gamma(s)$ intersects $\tilde\gamma(U)$ only at $\tilde\gamma(s)$. Thus connecting the end points of $\tilde\gamma(\ol U)$ to $o$ yields a simple closed curve, say $\Gamma$. We call the segments which run between $o$ and end points of $\tilde\gamma(\ol U)$ the sides of $\Gamma$, and let $\theta$ denote  the interior angle  of $\Gamma$ at $o$. We may assume that $U$ is so small that $\theta\leq\pi$. Then we claim that the region $K$ bounded by $\Gamma$ is convex, which will complete the proof.  
To this end let $p_0$, $p_1\in \inte (K)$. There exists a curve $p\colon [0,1]\to \inte(K)$ with $p(0)=p_0$, and $p_1=p(1)$, since $\inte(K)$ is path connected by Jordan curve theorem. Let $\ol t\in [0,1]$ be the supremum of all points $t\in[0,1]$ such that the line segment $p(0)p(t)\subset\inte(K)$. If $\ol t=1$, for all pairs of points $p_0$, $p_1\in \inte(K)$, then the line segment $p(0)p(1)\subset\inte(K)$. So $\inte(K)$ is convex, which implies that $K$ is convex, and we are done. 

Suppose  that $\inte(K)$ is not convex. Then $\ol t<1$ for some pair of points $p_0$, $p_1\in\inte(K)$. Note also that $\ol t>0$ since $p_0\in\inte(K)$. So an interior point $x$ of  $p(0)p(\ol t)$ intersects $\partial K=\Gamma$, while $p(0)p(\ol t)\subset K$.  Since $\theta\leq\pi$, $x$ cannot lie on a side of $\Gamma$, for then either $p(0)$ or $p(1)$ will be forced to lie on a side of $\Gamma$ as well, which is not possible as they are interior points of $K$. So $x$ must lie on $\tilde\gamma(U)$. Now we may slightly perturb the segment $p(0)p(\ol t)$ so that a point of it leaves $K$ while its end points remain in $\inte(K)$. Then we obtain a line segment $\sigma$ whose end points lie on $\tilde\gamma(U)$ while its interior lies outside $K$. Thus if we replace the segment of $\tilde \gamma(U)$ which lies between the end points of $\sigma$ with the line segment $\sigma$, we obtain
 a star-shaped curve $\tilde\beta$ with $L(\tilde\beta)<L(\tilde\gamma)$. 
Parameterize $\tilde\beta$ by letting $\tilde\beta(t)$ be  the point where the ray generated by $\tilde\gamma(t)$ intersects $\tilde\beta$. Then $|\tilde\beta(t)|\geq|\tilde\gamma(t)|$.
  Now set 
$
\beta(t):=\frac{|\tilde \beta(t)|}{|\tilde\gamma(t)|}\gamma(t).
$
 Then $\tilde\beta$ is the unfolding of $\beta$. So
$
L(\beta)=L(\tilde\beta)<L(\tilde\gamma)=L(\gamma).
$
 On the other hand, $o\in\conv(\beta)$; otherwise there exists $u\in\S^2$ such that $\langle\beta(t),u\rangle>0$ for all $t\in[a,b]$, which in turn yields that
$
\langle\gamma(t),u\rangle>0,
$
 which is not possible since $o\in\conv(\gamma)$. So $\lambda\beta(t)\in \conv(\beta)$ for all $0\leq\lambda\leq 1$. In particular $\gamma\subset\conv(\beta)$. It follows that 
 $
\conv(\beta)\supset \conv(\gamma)\supset \S^2.
 $
  Thus  $\beta$ is an inspection curve  shorter than $\gamma$, which is the desired contradiction.
\end{proof}

\section{Spiral Decomposition of the Unfolding}\label{sec:decomposition}
Using the local convexity property established in the last section, we will show here that  the unfolding of a minimal inspection curve admits
a partition into certain segments  we call spirals. Note that a locally convex curve  is rectifiable, and therefore may be reparameterized with constant speed $C$. Then $|\gamma'|=C$ at almost all differentiable points of $\gamma$ by Lemma \ref{lem:basic}. Furthermore, local convexity also ensures \cite[Lem. 5.1]{ghomi-wenk-arXiv2020}:

\begin{lem}\label{lem:gamma-speed}
Let $\gamma\colon[a,b]\to\R^2$ be a locally convex curve with constant speed $C$. Then one sided derivatives of $\gamma$ exist at all points. Furthermore, $|\gamma'_+(a)|=|\gamma'_-(b)|=|\gamma'_\pm(t)|=C$ for all $t\in(a,b)$.
\end{lem}

Let $\gamma\colon[a,b]\to\R^2\setminus\{o\}$ be a  locally convex curve with constant speed.  If $t\in(a,b)$ is a differentiable point of $\gamma$ then $\gamma'_+(t)=\gamma_-'(t)=\gamma'(t)$. Thus the above lemma shows that  $|\gamma'(t)|\neq 0$ at differentiable points of $\gamma$, since $C=L(\gamma)/(b-a)>0$. In particular the angle $\alpha$ is well defined at differentiable points of $\gamma$. We count $a$, $b$ among differentiable points of $\gamma$, and set
$\gamma'(a):=\gamma'_+(a)$, $\gamma'(b):=\gamma'_-(b)$.  We say that $\gamma$ is a \emph{spiral} if (i) $\gamma$ is locally convex with respect to $o$, (ii) $|\gamma|$ is nondecreasing, (iii) $\alpha(a)=\pi/2$, and (iv) $|\gamma(a)|\geq 1$. Note that condition (ii), via \eqref{eq:alpha-gamma-prime0}, implies that 
\be\label{eq:alpha-leq-pi/2}
\alpha(t)\leq \pi/2
\ee
at differentiable points $t\in[a,b]$ of $\gamma$. 
We say that $\gamma$ is a \emph{strict} spiral if $|\gamma|$ is increasing. By a \emph{spiral decomposition} of a constant speed curve $\gamma\colon[a,b]\to\R^2$ we mean a collection $U_i$ of mutually disjoint open subsets of $[a,b]$ such that (i)  $\gamma|_{\ol U_i}$ is a strict spiral, after switching the direction of $\gamma$ if necessary, and (ii) $|\gamma|'=0$ almost everywhere on $[a,b]\setminus \cup_i \ol U_i$. By a \emph{parameter shift}  we mean  replacing $t$ with $(t+x)\,\text{mod}\, (b-a)$ for some $x\in[a,b]$. The main result of this section is:

\begin{prop}\label{prop:decomposition}
Let $\gamma\colon[a,b]\to\R^3$ be a  minimal  inspection curve.  Then  the unfolding of $\gamma$  admits a spiral decomposition, after a parameter shift.
\end{prop}

We may assume that $\gamma$ has constant speed.
Let $\tilde\gamma$ be the unfolding of $\gamma$ and $x\in (a,b)$ be a local minimum point of the height function $|\gamma|=|\tilde\gamma|$. Then $\tilde\gamma$ is locally supported from below by a circle of radius $|\tilde\gamma(x)|$ centered at $o$. Thus, since $\tilde\gamma$ is locally convex with respect to $o$, there can pass only one local support line of $\tilde\gamma$ through $\tilde\gamma(x)$. Consequently $\tilde\gamma$ is differentiable at $x$ \cite[Thm. 1.5.15]{schneider2014}. Furthermore,  the local support line at $\tilde\gamma(x)$ must be orthogonal to $\tilde\gamma(x)$, since $x$ is a local minimum of $|\tilde\gamma|$. So $\langle \tilde\gamma'(x),\tilde\gamma(x)\rangle =0$. Now if we shift the parameter of $\tilde\gamma$ by $x$, it follows that $\tilde\gamma$ is orthogonal to $\tilde\gamma(a)$ and $\tilde\gamma(b)$. Thus to prove Proposition \ref{prop:decomposition} it suffices to show:
\begin{lem}\label{lem:decomposition}
Let $\gamma\colon[a,b]\to\R^2$ be a constant speed curve which is locally convex with respect to $o$. Suppose that $\alpha(a)=\pi/2=\alpha(b)$, and $|\gamma|\geq 1$. Then $\gamma$ admits a spiral decomposition.
\end{lem}
\noindent
To prove this lemma first recall that, as we had mentioned at the end of Section \ref{sec:integral}, the height function $|\gamma|$ of a constant speed curve is Lipschitz. In particular $|\gamma|$ is absolutely continuous and so it satisfies the fundamental theorem of calculus:
\be\label{eq:FTM}
 |\gamma(s)|-|\gamma(t)|=\int_t^s|\gamma|'(t)dt
 \ee
for every pair of points $s<t\in[a,b]$. Furthermore,  let us reiterate that by \eqref{eq:alpha-gamma-prime0} if the speed of $\gamma$ is equal to $C$ then
for almost every $t\in[a,b]$,
$\alpha(t)=\cos^{-1}(|\gamma|'(t)/C)$.

\begin{proof}[Proof of Lemma \ref{lem:decomposition}]
Let $X$ be the set of points $t\in[a,b]$ such that $\gamma$ has a local support line at $\gamma(t)$ which is orthogonal to $\gamma(t)$. Then it follows from \eqref{eq:alpha-gamma-prime0} that $|\gamma|'=0$ almost everywhere on $X$.
Also note that $X$ is closed, since the limit of any sequence of support lines of a convex body is a support line. Consequently each (connected) component $U$ of  $[a,b]\setminus X$ is an open subinterval of $[a,b]$. 
We claim that  $\gamma|_{\ol U}$ is a strict spiral, possibly after switching the direction of $\gamma|_{\ol U}$, which will complete the proof. 

To establish the above claim first note that by \eqref{eq:alpha-gamma-prime0}, $|\gamma|'$ cannot vanish at any differentiable point of $|\gamma|$ on $U$, for any such point would belong to $X$. We will show  that either $|\gamma|'>0$ almost everywhere on $U$ or else $|\gamma|'<0$ almost everywhere on $U$.  To  this end we start by orienting each local support line $\ell$ of $\gamma$ consistent with the orientation of $\gamma$ at the point of contact with $\ell$, so that the \emph{angle} between any support line $\ell$ and the position vector of its point of contact will be consistently  defined along $\gamma$. Now suppose, towards a contradiction, that there are subsets $X$ and $Y$ of $U$ with nonzero measure such that $|\gamma|'>0$ on $X$ and $|\gamma|'<0$ on $Y$. Since $X\cup Y$ is dense in $U$, there exists a point $r\in U$ which is a limit both of $X$ and $Y$. More specifically,
there are sequences of differentiable points $t_i$, $s_i$ converging to $r$ such that 
 $|\gamma|'(t_i)>0$ and $|\gamma|'(s_i)<0$, which in turn implies that $\alpha(t_i)<\pi/2$ and $\alpha(s_i)>\pi/2$ by \eqref{eq:alpha-gamma-prime0}. Hence, since the limit of support lines to a convex body  is a support line, there  exists  a local support line through $r$ which makes an angle $\leq\pi/2$ with $r$, and there also exists a support line through $r$ which makes an angle $\geq\pi/2$ with $r$. So we conclude that there  exists a local support line orthogonal to $r$, which is not possible by definition of $U$. Hence $|\gamma|'$ is always positive or always negative at differentiable points of $|\gamma|$ in $U$ as claimed. 

 Now it follows from \eqref{eq:FTM} that $|\gamma|$ is strictly monotone on $\ol U$. Next, let $t_0$ be the boundary point of $\ol U$ which forms the minimum point of $|\gamma|$ on $\ol U$. We have to show that $\gamma|_{\ol U}$ is orthogonal to $\gamma(t_0)$. If $t_0=a$, $b$ this already holds by assumption. So suppose that $t_0\in(a,b)$. Then $t_0\in X$, and so $\gamma$ has a local support line $\ell$ at $\gamma(t_0)$ which is orthogonal to $\gamma(t_0)$. Since $\gamma$ is locally convex with respect to $o$,  locally $\gamma$ lies below $\ell$. On the other hand, since $t_0$ is the minimum point of $|\gamma|$ on ${\ol U}$, then, near $\gamma(t_0)$, $\gamma|_{\ol U}$ lies above the circle $S$ with radius $|\gamma(t_0)|$ centered at $o$. Thus $\gamma|_{\ol U}$ must be orthogonal to $\gamma(t_0)$ as desired. We conclude then that $\gamma|_{\ol U}$ is a spiral, after switching the direction of $\gamma|_{\ol U}$ if necessary, so that $\gamma(t_0)$ becomes its initial point.
\end{proof}

\section{Efficiency of Line Segments}\label{sec:horizon}
Here we derive some formulas for the horizon and therefore efficiency of line segments, which may be checked using \cite{ghomi-wenk:Mathematica2}. Suppose that we have a  line segment $p_0p_1$ (with $p_0\neq p_1$) such that the line generated by $p_0p_1$ avoids  $\inte(B^3)$, see Figure \ref{fig:circles}. 
\begin{figure}[h]
\begin{overpic}[height=1.5in]{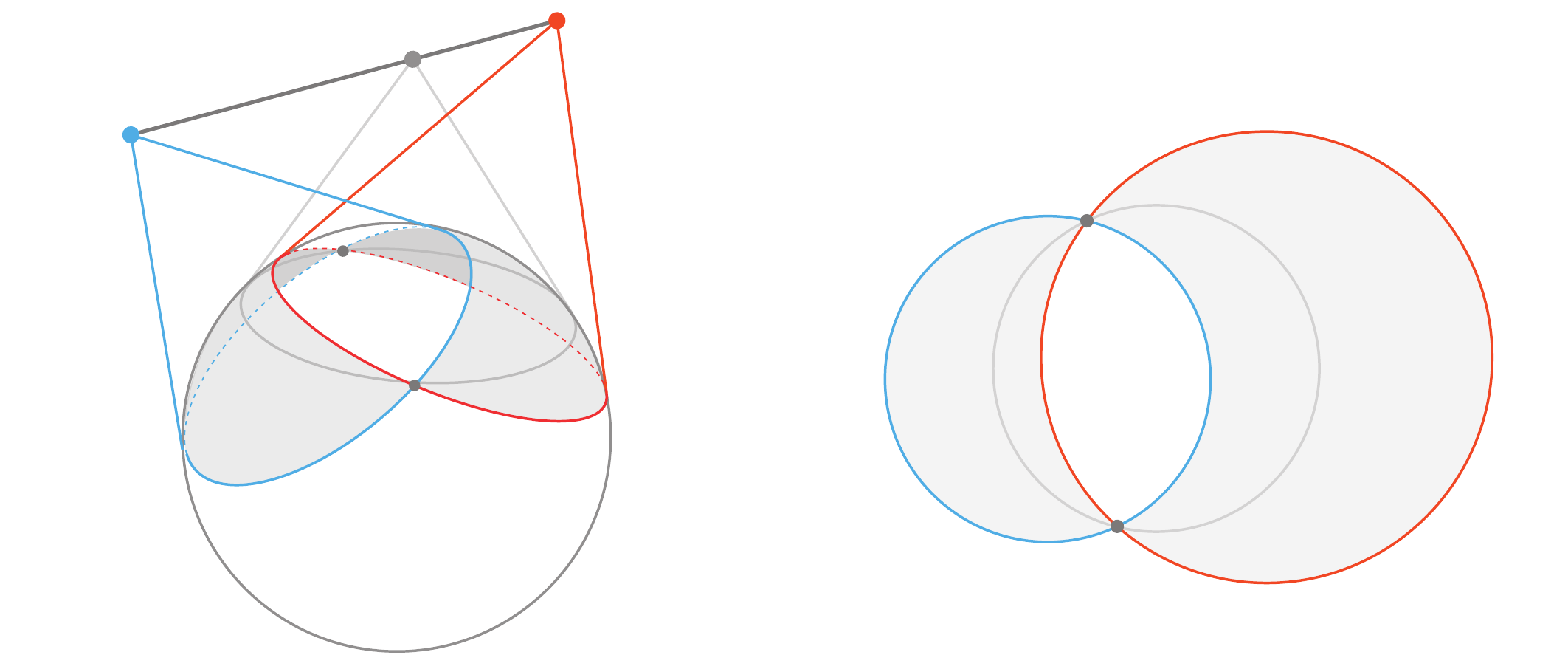}
\put(6,35.5){\small$p_0$}
\put(36,43){\small$p_1$}
\put(53,7){\small$H(p_0)$}
\put(91,7){\small$H(p_1)$}
\put(26,15){\small$q$}
\put(21,24){\small$q'$}
\put(70,6){\small$q$}
\put(68,30){\small$q'$}
\end{overpic}
\caption{}\label{fig:circles}
\end{figure}
For each point $p$ on $p_0p_1$, let $C_p$ be the (inspection) cone generated by all rays which emanate from $p$ and pass through a point of $B^3$. Let $H(p)$ be the set of points where $\partial C_p$ touches $\S^2$, i.e., the \emph{horizon circle} from the point of view of $p$. Then $H(p_0p_1)$ is the area of the union of all  horizon circles $H(p)$. Let $C_i:=C_{p_i}$, $H_i:=H(p_i)$, and  $\{q,q'\}:=H_0\cap H_1$; it is possible that $q=q'$ which happens precisely when the line through $p_0$ and $p_1$ is tangent to $\S^2$. Note that all horizon circles $H(p)$ pass through $q$ and $q'$, because the triangles $p_0p_1q$ and $p_0p_1q'$ lie on planes which are tangent to $\S^2$. Thus $H(p_0p_1)$ consists of the two lunar regions determined by $H_0$ and $H_1$, if $q\neq q'$; otherwise, $H(p_0p_1)$ is the region lying inside one of the circles and outside the other. More precisely, if $D_i$  denote the (inspection) disks in $\S^2$ bounded by $H_i$, which lie inside $C_i$, then
\be\label{eq:Hp0p1}
H(p_0p_1)=A(D_0)+ A(D_1)- 2 A(D_0\cap D_1),
\ee
where $A$ stands for area.
We may use basic spherical trigonometry to compute $H(p_0p_1)$ as follows. To start, note that if we set
$
h_i:=|p_i|, \;\text{and}\; \ell:=|p_0p_1|
$
then the radii of $D_i$ and the distance in $\S^2$ between the centers of $D_i$ are given respectively by
$
\rho_i:=\cos^{-1}(1/h_i), 
$
and
$
d:=\cos^{-1}((h_0^2+h_1^2-\ell^2)/(2h_0h_1)).
$
 It is a basic fact that
$
A(D_i)=4\pi\sin^2(\rho_i/2).
$
The formula for $A(D_0\cap D_1)$ is also known in terms of $\rho_i$ and $d$ \cite{tovchigrechko-vakser2001,ghomi-wenk-arXiv2020}.
Substituting these  formulas in  \eqref{eq:Hp0p1} yields the following formula for $H(p_0p_1)$:
\begin{small}
\begin{multline*}
H(h_0,h_1,\ell)
=
 4 \Big(\frac{1}{h_0}\sin ^{-1}\left(\frac{h_1^2-h_0^2-\ell^2}{\sqrt{\left(h_0^2-1\right)
   \left((h_0+h_1)^2-\ell^2\right) \left(\ell^2-(h_1-h_0)^2\right)}}\right)
   +\\
   \frac{1}{h_1}\sin
   ^{-1}\left(\frac{h_0^2-h_1^2-\ell^2}{\sqrt{\left(h_1^2-1\right)
   \left((h_0+h_1)^2-\ell^2\right) \left(\ell^2-(h_1-h_0)^2\right)}}\right)
   +\cos
   ^{-1}\left(\frac{h_0^2+h_1^2-\ell^2-2}{2 \sqrt{\left(h_0^2-1\right)
   \left(h_1^2-1\right)}}\right)\Big).
\end{multline*}
\end{small}
Note that
$
h_1=\sqrt{h_0^2+\ell^2+2 h_0\ell\cos(\alpha)},\;\text{where}\;  \alpha:=\angle(p_0,p_0p_1).
$
In particular, if $\alpha=\pi/2$ then we obtain a formula for the horizon of one-edge spirals:
\begin{multline}\label{eq:h0-ell}
 H(h_0,\ell)=
4 \left(\cos ^{-1}\left(\frac{\sqrt{h_0^2-1}}{\sqrt{h_0^2+\ell^2-1}}\right)
-\frac{1}{\sqrt{h_0^2+\ell^2}}\sin ^{-1}\left(\frac{\ell}{h_0
   \sqrt{h_0^2+\ell^2-1}}\right)\right).
\end{multline}
Finally, if we set $\ell=\sqrt{h_1^2-h_0^2}$ in the last expression we obtain another formula for the horizon of one-edge spirals in terms of $h_0$ and $h_1$:
\begin{equation*}\label{eq:h0-h1}
 H(h_0,h_1):=4 \left(\cos ^{-1}\left(\frac{\sqrt{h_0^2-1}}{\sqrt{h_1^2-1}}\right)-\frac{1}{h_1}\sin ^{-1}\left(\frac{\sqrt{h_1^2-h_0^2}}{h_0
   \sqrt{h_1^2-1}}\right)\right).
\end{equation*}
The graph of the corresponding efficiency function $E(h_0,h_1):=H(h_0,h_1)/\sqrt{h_1^2-h_0^2}$, for $h_1\geq h_0\geq 1$ is shown in Figure \ref{fig:line-eff}.
\begin{figure}[h]
\begin{overpic}[height=1.5in]{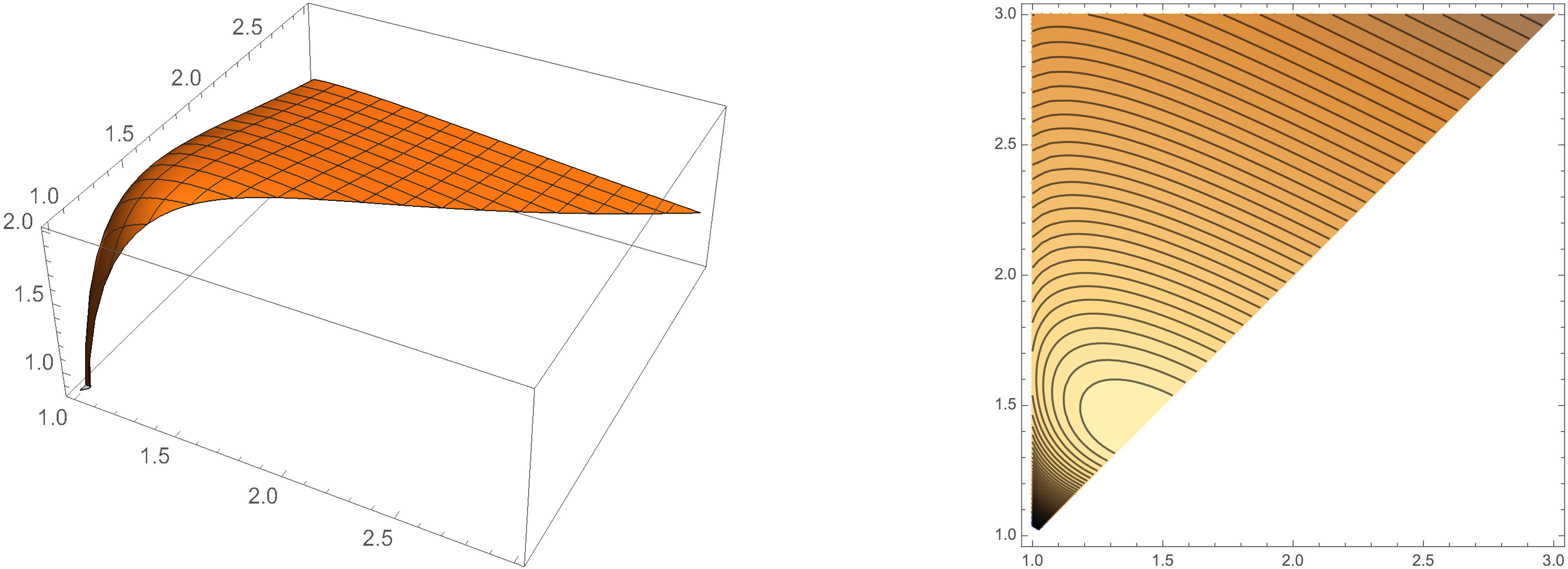}
\end{overpic}
\caption{}\label{fig:line-eff}
\end{figure}
Note that $E(h_0,h_1)\leq 2$, and equality holds only when $h_0=h_1=\sqrt 2$, or the spiral has constant height $\sqrt2$. Below we will prove that all spirals satisfy these properties.

\section{Upper Bound for Efficiency of Spirals}\label{sec:inequality}
Here we apply formulas of the last section to bound the efficiency of spirals  via a variational argument applied to polygonal curves. First we need to show that spirals form a locally compact space on which the efficiency functional is continuous. To this end we start by extending the definition of a spiral as follows. We say that $\gamma\colon[a,b]\to\R^2$ is a \emph{(generalized) spiral} provided that either $\gamma$ is a spiral as defined  in Section \ref{sec:decomposition}, or else $\gamma$ is a constant map with $|\gamma|\geq 1$. 
We also extend the definition of efficiency by setting
\begin{equation}\label{eq:LP0}
E(\gamma):=\frac{4\sqrt{|\gamma|^2-1}}{|\gamma|^2}, \quad\quad \text{when} \quad\quad L(\gamma)=0.
\end{equation}
So by Proposition \ref{prop:H}, when $L(\gamma)=0$, $E(\gamma)$ is the efficiency of a curve of constant height $|\gamma|$. Note  that then $E(\gamma)\leq 2$, and $E(\gamma)=2$ only when $|\gamma|=\sqrt2$. The space of spirals $\gamma\colon[a,b]\to\R^2$, with the topology induced on it by the uniform metric \eqref{eq:sup-norm}, will be denoted by $\mathcal{S}([a,b])$. To show that $E$ is continuous on $\mathcal{S}([a,b])$  first we observe that:

\begin{lem}\label{lem:alpha3}
Let $\gamma\colon[a,b]\to\R^2$ be a constant speed spiral with $L(\gamma)\neq 0$.  Then 
\be\label{eq:lem-alpha}
\alpha(t)\geq\sin^{-1}\big(|\gamma(a)|/|\gamma(t)| \big),
\ee
at all differentiable points $t\in[a,b]$ of $\gamma$.
\end{lem}
\begin{proof}
We may assume, for convenience, that the speed of $\gamma$ is $1$. Then
taking the cosine of both sides of \eqref{eq:lem-alpha} and squaring yields
\be\label{eq:lem-alpha2}
\langle\gamma,\gamma'\rangle^2\leq |\gamma|^2-r^2.
\ee
Recall that by \eqref{eq:alpha-leq-pi/2}, $\alpha(t)\leq\pi/2$ which in turn yields that $\langle\gamma,\gamma'\rangle\geq 0$. Thus
\eqref{eq:lem-alpha2} is equivalent to \eqref{eq:lem-alpha}.
To establish \eqref{eq:lem-alpha2}, first assume that $\gamma$ is $\C^{1,1}$. Then the left hand side of \eqref{eq:lem-alpha2} is Lipschitz;  therefore, it is differentiable almost everywhere and satisfies the fundamental theorem of calculus. Furthermore, since $\gamma$ is locally convex with respect to $o$, $\langle \gamma,\gamma''\rangle\leq 0$ almost everywhere. So, since $\langle\gamma'(a),\gamma(a)\rangle=0$, and $\langle\gamma,\gamma'\rangle\geq 0$,
$$
\langle\gamma(t),\gamma'(t)\rangle^2
=
2\int_a^t\langle \gamma(s),\gamma'(s)\rangle \big(1+\langle \gamma(s),\gamma''(s)\rangle\big)ds
\leq
2\int_a^t\langle \gamma(s),\gamma'(s)\rangle ds
=
 |\gamma(t)|^2-r^2,
$$
as desired. To establish the general case we consider the outer parallel curves $\gamma_\epsilon$ of $\gamma$ at distance $\epsilon>0$. These curves are given by setting $\gamma_\epsilon(a):=\gamma(a)+\epsilon \gamma(a)/|\gamma(a)|$, and requiring that $\gamma_\epsilon$ maintain constant distance $\epsilon$ from $\gamma$. Since $\gamma$ is locally convex, $\gamma_\epsilon$ is $\C^{1,1}$ \cite[Prop. 2.4.3]{hormander}. Furthermore, it is not difficult to see that $\gamma_\epsilon$ is a spiral. So $\gamma_\epsilon$ satisfies \eqref{eq:lem-alpha2}. Next note that
for each differentiable point $\gamma(t)$ of $\gamma$ there exists a unique point $\gamma_\epsilon (t_\epsilon)$ of $\gamma_\epsilon$ which is closest to $\gamma(t)$. Then $\alpha(t)=\alpha_\epsilon(t_\epsilon)$ where $\alpha_\epsilon:=\angle(\gamma_\epsilon,\gamma'_\epsilon)$.
Thus
$
\alpha(t)=\alpha_\epsilon(t_\epsilon)\geq\sin^{-1}\left(\frac{r+\epsilon}{|\gamma_\epsilon(t_\epsilon)|}\right).
$
Letting $\epsilon\to 0$ completes the proof.
 \end{proof}

Since for a spiral $\gamma(t)$ with $L(\gamma)\neq 0$, $\alpha(t)\leq\pi/2$, the last lemma shows that 
$\alpha(t)\to\pi/2$ as $\gamma(t)\to \gamma(a)$. This observation, together with  some basic convex analysis, yields:

\begin{lem}\label{lem:continuous}
The efficiency functional $E$ is continuous on the space of spirals $\mathcal{S}([a,b])$.
\end{lem}
\begin{proof}
For convenience we may assume that $[a,b]=[0,1]$.
Let $\gamma_k\colon[0,1]\to\R^2$ be a sequence of spirals converging to a spiral 
$\gamma\colon[0,1]\to\R^2$.  We have to show that $E(\gamma_k)\to E(\gamma)$. To this end, we may assume that all spirals have constant speed. First suppose that $L(\gamma)=0$. If $L(\gamma_k)=0$ as well, then we are done by \eqref{eq:LP0}. So we may assume that $L(\gamma_k)>0$, by passing to a subsequence.  Then, by Proposition \ref{prop:H}
\begin{multline}\label{eq:HP^k}
E(\gamma_k)
=\int_0^{1}\int_0^{2\pi} \frac{1}{|\gamma_k(t)|^2}\left| \sqrt{|\gamma_k(t)|^2-1}\sin\big(\alpha_k(t)\big)\cos(\theta)+\cos\big(\alpha_k(t)\big) \right|d\theta dt.
\end{multline}
 Note that $\gamma_k(a)\to\gamma(a)$ and $\gamma_k(t)\to\gamma(t)=\gamma(a)$.  So $\gamma_k(t)\to\gamma_k(a)$. Consequently, by Lemma \ref{lem:alpha3}, $\alpha_k(t)\to \pi/2$. So, since the integrand in \eqref{eq:HP^k} is bounded,  the dominated convergence theorem yields that
$$
E(\gamma_k) \to\int_0^1\int_0^{2\pi}\frac{\sqrt{|\gamma(a)|^2-1}}{|\gamma(a)|^2}|\cos(\theta)|\,d\theta dt=4\frac{\sqrt{|\gamma(a)|^2-1}}{|\gamma(a)|^2}=E(\gamma),
$$
as desired.
Next suppose that $L(\gamma)>0$, then we may assume that $L(\gamma_k)>0$ as well. So, again \eqref{eq:HP^k} holds.
By assumption $\gamma_k\to\gamma$ uniformly.  Furthermore, since $\gamma$ and $\gamma_k$ are locally convex, it follows that $\gamma'_k\to\gamma'$ almost everywhere on $[0,1]$. This can be shown by representing $\gamma_k$, $\gamma$ locally as graphs of  convex functions and applying well-known results on convergence of derivatives from classical convexity theory; e.g., see \cite[C(9), p. 20]{roberts-varberg}, \cite[Lem. 2]{tsuji1952}, or \cite{lackovic1982}. 
So $\alpha_k\to\alpha$ almost everywhere on $[0,1]$. Thus by the dominated convergence theorem
\begin{multline*}
E(\gamma_k) \to 
 \int_0^{1}\int_0^{2\pi} \frac{1}{|\gamma(t)|^2}\left| \sqrt{|\gamma(t)|^2-1}\sin\big(\alpha(t)\big)\cos(\theta)+\cos\big(\alpha(t)\big) \right|d\theta dt=E(\gamma),
\end{multline*}
which completes the proof.
\end{proof}

A \emph{polygonal curve} $P$  is a collection of line segments determined by a sequence of points $p_0,\dots, p_m\in\R^2$ with $p_{i+1}\neq p_{i}$. We also allow $P$ to be a single point, and use the formal notation $P=(p_0,\dots, p_m)$ to specify a polygonal curve. The line segments  $p_ip_{i+1}$  are called the \emph{edges} of $P$. Each polygonal curve $P$ admits a unique constant speed parameterization $\gamma_P\colon[0,1]\to P$, with $\gamma_p(0)=p_0$ which traces the edges of $P$.  The \emph{distance} between a pair of polygonal curves $P^1$, $P^2$ is defined as 
$
\dist\big(\gamma_{P^1},\gamma_{P^2}\big),
$
the metric given by \eqref{eq:sup-norm}.
 Let $\P^m$ denote the space of polygonal curves with at most $m$ edges in $\R^2$, endowed with the topology induced by $\dist$. Then  $\P^m$  is locally compact.
We say that $P\in\P^m$ is a \emph{polygonal spiral}  provided that $\gamma_P$ is a spiral. Let $\mc{S}^m$ be the collection of polygonal spirals with at most $m$ edges.  Lemma \ref{lem:alpha3} together with 
Blaschke's selection principle \cite[Thm. 7.3.8]{bbi:book} quickly yields \cite[Lem. 7.3]{ghomi-wenk-arXiv2020}:

\begin{lem}\label{lem:spiral-compact}
The space of polygonal spirals $\mc{S}^m$ is locally compact, for every $m\geq 0$.
\end{lem}

Next we observe that

\begin{lem}\label{lem:inequality}
The efficiency of any polygonal spiral is at most $2$.
\end{lem}
\begin{proof}
Fix an integer $m\geq 0$, and number $R>1$. By Lemma \ref{lem:spiral-compact} there exists a polygonal spiral  $P=(p_0,\dots, p_k)$ which maximizes $E$ among elements of $\mathcal{S}^m$ which lie in the ball of radius $R$ centered at $o$. We need to show that $E(P)\leq 2$. If $P$ is a singleton, this is guaranteed by \eqref{eq:LP0}. So we may assume that $k\geq1$. Then $r:=|p_0|<R$.  Note that for $-\epsilon< t<\epsilon$ there exists a point $p_0^t$ such that $p_0^t$ is orthogonal to $p_0^tp_1$, and $|p_0^t|=(1+t)r$, assuming that $\epsilon$ sufficiently small. 
Indeed $p_0^t$ lies on an arc of the circle of radius $|p_1|/2$ which is centered at the midpoint of $op_1$; see Figure \ref{fig:perturb}.
\begin{figure}[h]
\begin{overpic}[height=1.1in]{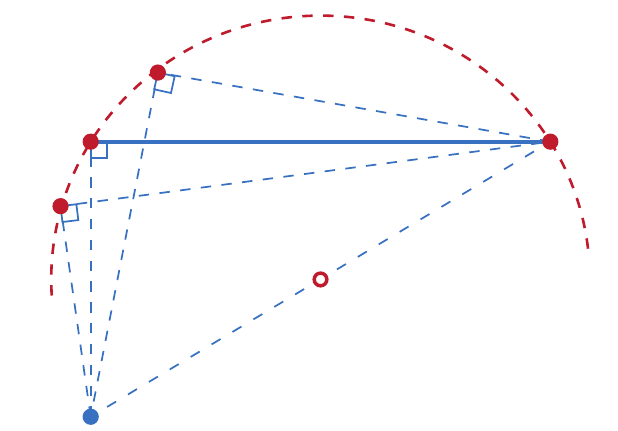}
\put(5,46){\small$p_0$}
\put(92,46){\small$p_1$}
\put(9,-1){\small$o$}
\end{overpic}
\caption{}\label{fig:perturb}
\end{figure}
Furthermore, choosing $\epsilon$ sufficiently small, we can ensure that $P^t:=(p_0^t,p_1,\dots, p_k)$ is locally convex. Thus $P^t$ will be a spiral provided that $|p_0^t|\geq 1$, which will be the case for small $\epsilon$ provided that $r>1$ or else $t\geq 0$. Let us assume first that $r>1$. Then $P^t$ will be a spiral in the ball $B_R(o)$ for $-\epsilon<t<\epsilon$. Let $L(t)$, $H(t)$, and $E(t)$ denote respectively the length, horizon, and efficiency of $P^t$. Then
$
0= E'(0)=(H'(0)L(0)-H(0)L'(0))/L(0)^2,
$
which in turn yields
$
H'(0)/L'(0)=H(0)/L(0)=E(P).
$
To compute $L'(0)$ note that
$$
L(t)=L(p_0^t p_1)+L\big((p_1,\dots, p_k)\big)=\sqrt{|p_1|^2-(r+t)^2}+L\big((p_1,\dots, p_k)\big).
$$
So it follows that
$
L'(0)=-r/\sqrt{|p_1|^2-r^2}.
$
Next, to compute $H'(0)$, note that 
$
H(t)=H(p_0^t p_1)+H\big((p_1,\dots, p_k)\big).
$
Furthermore, by \eqref{eq:h0-ell} we have
$
H(p_0^t p_1)=H\big(r+t,L(p_0^t p_1)\big).
$
Now a computation \cite{ghomi-wenk:Mathematica2} yields that
\be\label{eq:H'0}
H'(0)=\frac{d}{dt} H\big(r+t,L(p_0^t p_1)\big)\Big|_{t=0}=-\frac{4}{r}\frac{\sqrt{r^2-1}}{\sqrt{|p_1|^2-r^2}}.
\ee
So we conclude that
$
E(P)=\frac{H'(0)}{L'(0)}=4\frac{\sqrt{r^2-1}}{r^2}\leq 2,
$
as desired. It remains to show that our earlier assumption  that $r>1$ was justified. Suppose then, towards a contradiction, that $r=1$. Then $E(t)$ and $H(t)$ will still be well defined for $t\geq 0$, and so will their right-hand derivatives at $0$. By \eqref{eq:H'0}, 
$H'_+(0)=0$. Thus
$$
E'_+(0)=\frac{-H(0)L'(0)}{L(0)^2}=\frac{H(0)}{L(0)^2}\frac{1}{\sqrt{|p_1|^2-1^2}}>0.
$$
So $E(t)> E(0)$, or  $E(P^t)> E(P)$, for small $t>0$ which is the desired contradiction. 
\end{proof}

Lemma \ref{lem:inequality} together with Lemma \ref{lem:continuous}  yields the main result of this section via a polygonal approximation:

\begin{prop}\label{prop:inequality}
The efficiency of any spiral is at most $2$.
\end{prop}

Proposition \ref{prop:inequality} together with Proposition \ref{prop:decomposition} establishes  \eqref{eq:main}. The rest of this work will be concerned with characterizing the case of equality in \eqref{eq:main}.

\section{Instantaneous Efficiency}\label{sec:geq-sqrt2}
Here we investigate another method for bounding the efficiency of spirals via a notion  first used in the proof of Proposition \ref{prop:H}.
Let $\gamma\colon[a,b]\to\R^3$ be a constant speed curve with $|\gamma|\geq 1$, and $t\in[a,b]$ be a differentiable point of $\gamma$ with $|\gamma'(t)|\neq 0$.  We define the  \emph{instantaneous efficiency} of $\gamma$ at  $t$ as
$$
E_\gamma(t):=\int_0^{2\pi} \big|F(|\gamma(t)|,\alpha(t),\theta)\big|d\theta,
$$
where
$$
F(h,\alpha,\theta):=\frac{1}{h^2}\left(\sqrt{h^2-1}\sin(\alpha)\cos(\theta)+\cos(\alpha)\right).
$$
If $t$ is a differentiable point of $t\mapsto H(\gamma |_{[a,t]})$, then by \eqref{eq:H2}
$
E_\gamma(t)=\dt H\left(\gamma |_{[a,t]}\right).
$
So $E_\gamma(t)$ is the rate of change of horizon along $\gamma$. Furthermore, by Proposition \ref{prop:H},
$
E(\gamma)=\frac{1}{b-a}\int_a^b E_\gamma(t)dt\leq \sup_{[a,b]} E_\gamma(t).
$
Thus to find an upper bound for $E(\gamma)$ it suffices to bound $E_\gamma$. To this end we compute the  above integral as follows. Let
$$
\Omega:=\{(h,\alpha)\,|\, h\geq 1,\; \sin^{-1}\left(1/h\right)\leq\alpha\leq \pi/2\}
$$
be the phase space of  possible values for  $(|\gamma(t)|,\alpha(t))$ at differentiable points of  curves $\gamma$ which lie outside $\S^2$ and whose tangent lines avoid $\inte(B^3)$. For every $(h,\alpha)\in\Omega$ we set
$
\theta_0=\theta_0(h,\alpha):=\cos ^{-1}(-\cot (\alpha )/\sqrt{h^2-1}).
$
Since $\sin(\alpha)\geq 1/h$,  $|\cot(\alpha)/\sqrt{h^2-1}|\leq 1$. So $\theta_0$ is well defined. Also note that  $F(h,\alpha, \pm\theta_0)=0$.
Now we may compute that \cite{ghomi-wenk:Mathematica2}
\begin{gather*}
\mathcal{E}(h,\alpha):=\int_0^{2\pi}|F(h,\alpha,\theta)|\,d\theta
=\int_{-\theta_0}^{\theta_0} F(h,\alpha,\theta)\,d\theta-\int_{\theta_0}^{2\pi-\theta_0} F(h,\alpha,\theta) d\theta\\
=\frac{4}{h^2} \left(\sqrt{h^2 \sin ^2(\alpha )-1}+\cos (\alpha ) \sin ^{-1}\left(\frac{\cot (\alpha
   )}{\sqrt{h^2-1}}\right)\right).
\end{gather*}
Then 
$
E_\gamma(t)=\mathcal{E}\big(|\gamma(t)|,\alpha(t)\big).
$
For any set $X\subset[a,b]$ with measure $\mu(X)\neq 0$ we define
$
E(\gamma\big|_X):=\frac{1}{\mu(X)}\int_X E_\gamma(t)dt.
$
So we may record that
\begin{prop}\label{prop:IE}
Let $\gamma\colon[a,b]\to\R^3$ be a constant speed curve with $|\gamma|\geq 1$. If tangent lines of $\gamma$ avoid $\inte(B^3)$, then for any set $X\subset[a,b]$ with nonzero measure
$$
E(\gamma\big|_X)=\frac{1}{\mu(X)}\int_X\mathcal{E}\big(|\gamma(t)|,\alpha(t)\big) dt.
$$
\end{prop}
 The values of $\mathcal{E}$ on the phase space $\Omega$, which range from $0$ to about $2.6$, are shown in Figure \ref{fig:inst-eff}.
\begin{figure}[h]
\begin{overpic}[height=1.5in]{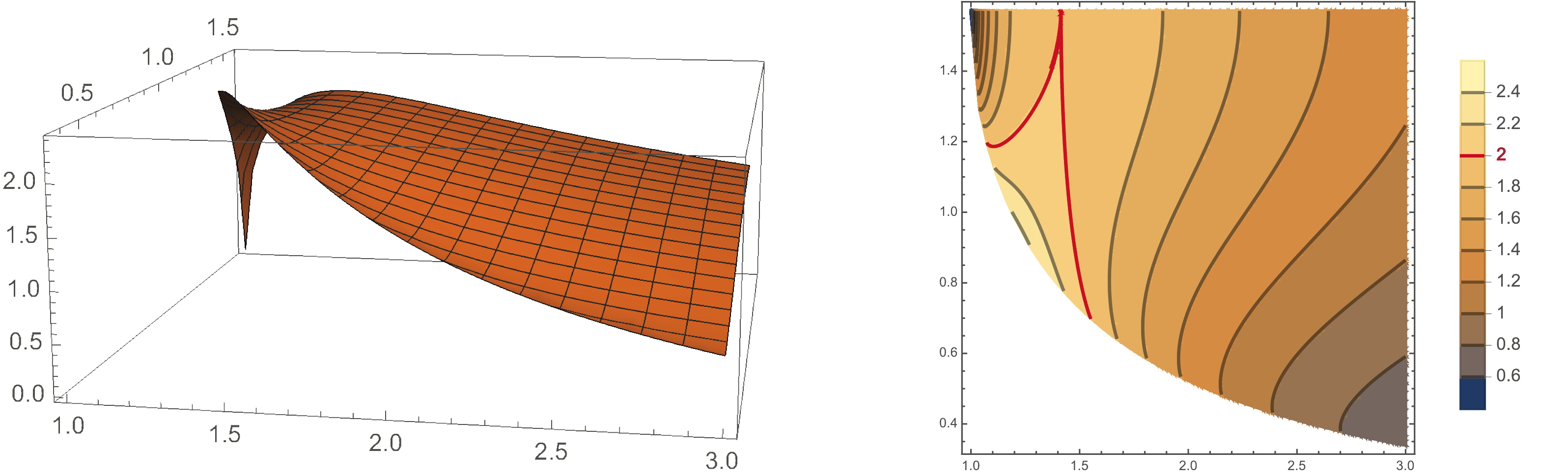}
\end{overpic}
\caption{}\label{fig:inst-eff}
\end{figure}
 Since $\mathcal{E}$ may exceed $2$, it is not possible to obtain the bound $E\leq 2$ for all spirals by bounding $\mathcal{E}$; however, spirals with initial height $\geq\sqrt2$ are special.
Set
$$
\mathcal{E}(h):=\mathcal{E}\left(h,\frac{\pi}{2}\right)=4\frac{\sqrt{h^2 -1}}{h^2}\leq 2.
$$
Note that $\mathcal{E}(h)=2$ only if $h=\sqrt2$.

\begin{lem}
Let $\gamma\colon[a,b]\to\R^2$ be a spiral with initial height $r\geq \sqrt{2}$. Then the instantaneous efficiency $E_\gamma(t)\leq \mathcal{E}(r)\leq 2$ for all $t\in[a,b]$.
\end{lem}
\begin{proof}
Recall that $E_\gamma(t)=\mathcal{E}(|\gamma(t)|,\alpha(t))$.
By Lemma \ref{lem:alpha3}, $\sin^{-1}(\sqrt2/r)\leq\alpha\leq \pi/2$. Since $\mathcal{E}(h,\pi/2)=\mathcal{E}(h)$, it suffices to check that $\alpha\mapsto \mathcal{E}(h,\alpha)$ is nondecreasing for $h\geq\sqrt{2}/\sin(\alpha)$. So we compute that \cite{ghomi-wenk:Mathematica2}
\begin{eqnarray*}
\frac{\partial \mathcal{E}}{\partial \alpha}(h,\alpha) 
&=&\frac{4}{h^2}\left(\cot (\alpha ) \sqrt{h^2 \sin ^2(\alpha )-1}-\sin (\alpha ) \sin ^{-1}\left(\frac{\cot (\alpha
   )}{\sqrt{h^2-1}}\right)\right)\\
   &\geq&  
    \frac{4}{h^2}\left(  \frac{\cos (\alpha )}{\sqrt{1-\cos ^2(\alpha )}}-\sin ^{-1}\left(\frac{\cos (\alpha )}{\sqrt{1+\cos
   ^2(\alpha )}}\right)\right).
\end{eqnarray*}
It remains to check that
$
\frac{x}{\sqrt{1-x^2}}-\sin ^{-1}\left(\frac{x}{\sqrt{1+x^2}}\right)\geq 0,
$
for $0\leq x\leq 1$.
Indeed this expression vanishes for $x=0$, and its derivative 
$
1/(1-x^2)^{3/2}-1/(x^2+1)
$ 
is nonnegative on $0\leq x\leq 1$. 
\end{proof}

\begin{cor}
Let $\gamma\colon[a,b]\to\R^2$ be a constant speed spiral with initial height $r \geq\sqrt{2}$. Then 
$$
E(\gamma)\leq \frac{1}{b-a}\int_a^b\mathcal{E}(|\gamma(t)|)dt\leq \mathcal{E}(r)\leq 2.
$$
Equality in the second inequality holds only when $|\gamma|\equiv r$, and equality in the third inequality holds only when $r=\sqrt2$.
\end{cor}

\section{Spirals with Maximum Efficiency}\label{sec:leq-sqrt2}
Here we refine the variational method employed in Section \ref{sec:inequality} to show that the efficiency of any spiral  assumes its maximum value  only when it has constant height $\sqrt{2}$.  We start by considering one edge  spirals $P=(p_0,p_1)$.
By a \emph{lifting} of $P$ we mean any polygonal curve $\tilde P=(\tilde p_0 , p_1)$ 
where $\tilde p_0=\lambda p_0$ for $\lambda> 1$.

\begin{lem}\label{lem:lifting}
Let $P=(p_0 , p_1)$ be a  spiral. For any  lifting $\tilde P$ of $P$,
$
H(P)<H(\tilde P).
$
\end{lem}
\begin{proof}
Suppose  $\tilde P:=(\tilde p_0, p_1)$, see Figure \ref{fig:cones}. 
\begin{figure}[h]
\centering
\begin{overpic}[height=1.75in]{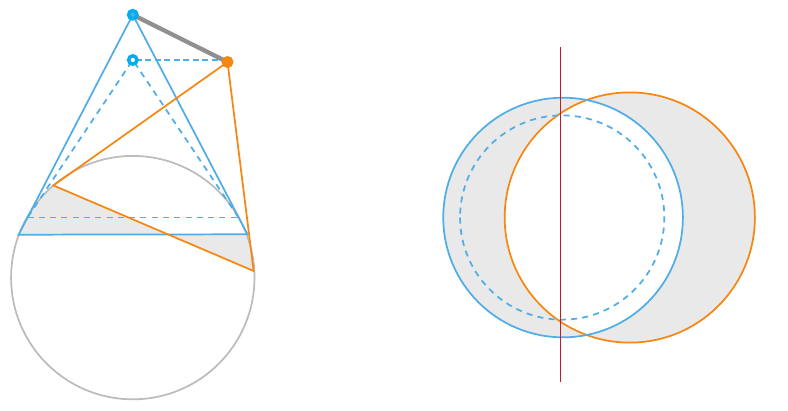}
\put(11.5,44){\small$p_0$}
\put(30.5,43.5){\small$p_1$}
\put(13,51.5){\small$\tilde p_0$}
\put(96,17){\small$H(p_1)$}
\put(45,17){\small$H(\tilde p_0)$}
\put(73,25){\small$H(p_0)$}
\end{overpic}
\caption{}\label{fig:cones}
\end{figure}
Since $p_0p_1$ is orthogonal to $p_0$, the horizon circle $H(p_1)$ (depicted in orange) bisects the horizon circle $H(p_0)$ (depicted in dotted blue line), because the two planes which contain $p_0p_1$ and are tangent to $\S^2$ intersect  $H(p_0)$ at a pair of its antipodal points. Thus the area that is gained by the horizon, as $p_0$ rises to $\tilde p_0$ exceeds the area which is lost.
\end{proof}

Let $P=(p_0,p_1)$ be a spiral and set $r:=|p_0|$, $R:=|p_1|$. For every $h\in[r,R]$, let 
$
\tilde P^h:=(\tilde p_0^{\,h},p_1)
$ 
be the lifting of $P$ such that the distance of $\tilde p_0^{\,h}p_1$ to $o$ is equal to $h$. Let $q^h$ be the closest point of  $\tilde p_0^{\,h}p_1$ to $o$; see Figure \ref{fig:lift-split}. 
\begin{figure}[h]
\begin{overpic}[height=1.2in]{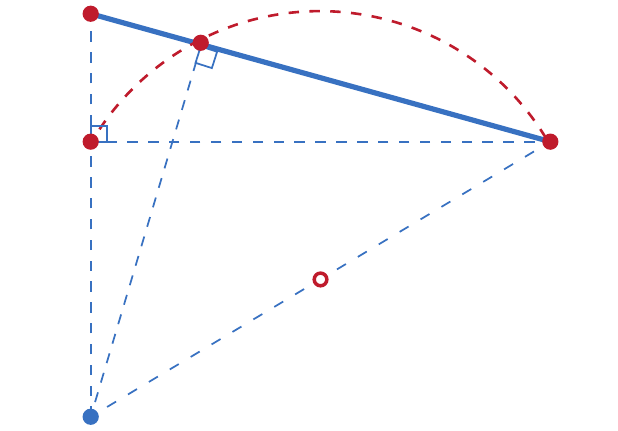}
\put(6,45){\small$p_0$}
\put(92,45){\small$p_1$}
\put(5.5,68){\small$\tilde p_0^{\,h}$}
\put(31,67){\small$q^h$}
\put(9,-1){\small$o$}
\end{overpic}
\caption{}\label{fig:lift-split}
\end{figure} 
Set
$
\tilde P^h_+:=(q^h,p_1)\;\text{and}\; \tilde P^h_-:=(\tilde p_0^{\,h},q^h).
$ 
\begin{lem}\label{lem:wt}
Let $P=(p_0,p_1)$ be a spiral with  initial height $r$, and final height $R$. Then for every $r\leq\rho\leq R$,
$$
H(P) \le  \int_{r}^{\rho} w(h) \mathcal{E}(h) dh + H(\tilde P^\rho_+),
$$
where  $\int_r^\rho w(h)dh=L(P)-L(\tilde P^\rho_+)$, and $w\geq r/\sqrt{R^2-r^2}$.
\end{lem}
\begin{proof}
 If  the desired inequality holds for all $r>1$, then it also holds for $r=1$ by continuity. So we may  assume that $r>1$.
We claim that
$$
w(h):=-\frac{d}{ds} L(\tilde P^s_+)\Big|_{s=h} =-\frac{d}{ds} \sqrt{R^2-s^2}\Big|_{s=h}=\frac{h}{\sqrt{R^2-h^2}}
$$
is the desired weight function. Clearly $\int_r^\rho w(h)dh=L(P)-L(\tilde P^\rho_+)$ and $w\geq r/\sqrt{R^2-r^2}$. Set
$
\Delta h:=(\rho-r)/n, \; h_i:=r+i\Delta h,\; \text{and} \;q^i:=q^{h_i}.
$
We define a sequence of liftings  as follows. Set $P^0:=P$. Once $P^i$ is defined, let $\tilde p_0^i$ be its initial point, $q^i$ be its closest point to $o$, and  set
$P^{i}_-:=(\tilde p_0^{\,i},q^i),$
$P^{i}_+:=(q^i,p_1)$. Then we define
$
P^{i+1}:=\tilde{(P^i_+)}^{h_{i+1}};
$
see Figure \ref{fig:steps}. 
 \begin{figure}[h]
\begin{overpic}[height=1.5in]{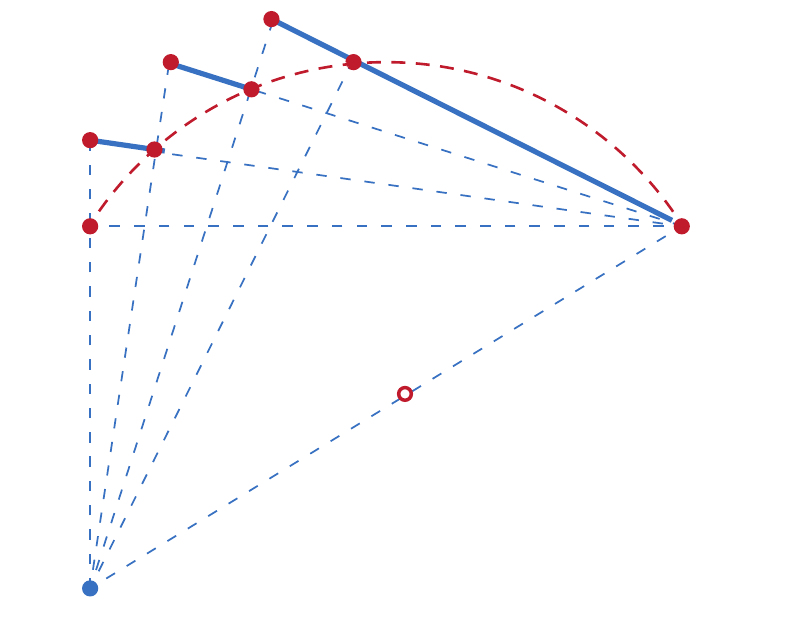}
\put(5,1){\small$o$}
\put(2,50){\small$p_0$}
\put(89,50){\small$p_1$}
\end{overpic}
\caption{}\label{fig:steps}
\end{figure}
By Lemma \ref{lem:lifting}  $H(P^i_+)<H(P^{i+1})=H( P^{i+1}_-)+H(P^{i+1}_+).
$
Applying this inequality iteratively yields
\begin{eqnarray*}
H(P) 
\leq  \sum_{i=1}^{n} H(P^i_-)+H(\tilde P^\rho_+).\\
\end{eqnarray*}
Now, for $0\leq s\leq  \Delta h$, let $q^i(s):= q^{h_{i-1}+s}$; see Figure \ref{fig:triangle}.
 \begin{figure}[h]
\begin{overpic}[height=0.6in]{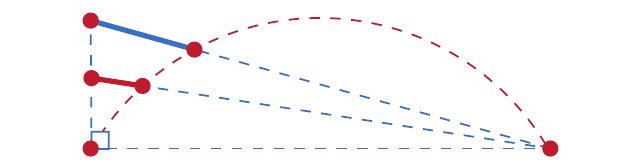}
\put(9,-3.5){\small$q^{i-1}$}
\put(35,17){\small$q^{i}$}
\put(7,24){\small$\tilde p_0^{\,i}$}
\put(0,12.25){\small$x^i(s)$}
\put(25,9.5){\small$q^i(s)$}
\put(91.5,0){\small$p_1$}
\end{overpic}
\caption{}\label{fig:triangle}
\end{figure}
Furthermore let $x^i(s)$ be the point where the line passing through  $q^i(s)$ and $p_1$ intersects  $q^{i-1}\tilde p_0^{\,i}$. Set $\sigma^i_s:=x^i(s)q^i(s)$. 
Let $f_i(s):=H(\sigma^i_s)$. 
By \eqref{eq:h0-ell},
$
f_i(s)=H\big(h_{i-1}+s,L(\sigma_i(s))\big).
$
So $f_i$  is $\C^\infty$ on $[0,\Delta h]$  provided that $h_{i-1}+s>1$ or $h_{i-1}>1$, which is  the case since $r>1$.
We have
$
H( P^i_-)=f_i(\Delta h)-f_i(0)\leq f_i'(0)\Delta h +C_i (\Delta h)^2,
$
where $C_i:=\sup_{[0,\Delta h]} f_i''(s)/2<\infty$. Note that $C_i$ depends continuously on $q^i$. So $C_i$ are bounded above by some constant $C$, independent of $i$, which yields
$$
H( P^i_-)\leq f_i'(0)\Delta h +C (\Delta h)^2= f_i'(0)\Delta h +C \frac{(\rho-r)^2}{n^2}.
$$
Next we compute that
$$
f_i'(0)=\frac{d}{ds}L( \sigma^i_s) \Big|_{s=0} E( \sigma^i_0)+L( \sigma^i_0) \frac{d}{ds}E( \sigma^i_s)\Big|_{s=0}
=\frac{d}{ds}L( \sigma^i_s) \Big|_{s=0} \mathcal{E}(h_{i-1}).
$$
If we let $\tau^i_s:=q^i(s)p_1$ then at $s=0$,
$
\frac{d}{ds}L( \sigma^i_s)+\frac{d}{ds}L(\tau^i_s) =\frac{d}{ds} |x^i(s)p_1|=0,
$
because $x^i(0)p_1=q^{i-1}p_1$ is orthogonal to $q^{i-1}\tilde p_0^{\,i}$. Now recall that $q^h p_1=\tilde P^h_+$. Thus
$$
\frac{d}{ds}L( \sigma^i_s) \Big|_{s=0}=-\frac{d}{ds}L(\tau^i_s) \Big|_{s=0}
=-\frac{d}{ds}L(\tilde P^{h_{i-1}+s}_+) \Big|_{s=0}
=w(h_{i-1}).
$$
The last four displayed expressions yield
$$
\sum_{i=1}^{n} H(P^i_-)\leq\sum_{i=0}^{n-1}\left( w(h_i)\mathcal{E}(h_i)\Delta h +C\frac{(\rho-r)^2}{n^2}\right)
=\sum_{i=0}^{n-1}w(h_i)\mathcal{E}(h_i)\Delta h +C\frac{(\rho-r)^2}{n}.
$$
Letting $n\to\infty$ 
 completes the proof.
\end{proof}

The last lemma via an induction yields:

\begin{lem}\label{lem:leqSqrt2}
Let $P$ be a polygonal spiral with  initial height $r$, and final height $R$. Then 
\be\label{eq:leqSqrt2}
H(P) \le \int_{r}^{R} w(h) \mathcal{E}(h) dh,
\ee
where $\int_r^R w(h)dh=L(P)$, and $w\geq r/\sqrt{R^2-r^2}$.
\end{lem}
\begin{proof}
If $P$ has only one edge \eqref{eq:leqSqrt2} holds by Lemma \ref{lem:wt}. Suppose  that \eqref{eq:leqSqrt2} holds for spirals with $n$ edges and let $P=(p_0,\dots, p_{n+1})$.  Let $\rho$ be the distance of the line spanned by $p_1p_2$ from $o$ and $q$ be the closest point of that line to the origin. Then $P':=(q, p_2,\dots, p_m)$  is a spiral with $n$ edges. 
Note that
$
H(P)=H\big((p_0,p_1)\big)+H\big((p_1,\dots, p_{n+1})\big),
$
and by Lemma \ref{lem:wt}, 
$
H\big((p_0,p_1)\big)\leq \int_r^\rho w_0(h)\mathcal{E}(h)dh +H\big((q,p_1)).
$
Thus
$$
H(P)\leq \int_r^{\rho} w_0(h)\mathcal{E}(h)dh +H(P'),
$$
where $\int_r^\rho w_0(h)dh=L(p_0p_1)-L(q p_1)=L(P)-L(P')$, and $w_0\geq r/\sqrt{|p_1|^2-r^2}\geq r/\sqrt{R^2-r^2}$.
By the inductive hypothesis
$$
H(P')\leq \int_{\rho}^R w_1(h)\mathcal{E}(h)dh,
$$
where $\int_r^\rho w_1(h)dh=L(P')$, and $w_1\geq \rho/\sqrt{R^2-\rho^2}\geq r/\sqrt{R^2-r^2}$.
Set $w:=w_0$ for $h<\rho$ and  $w:=w_1$ for $h\geq\rho$. Then the last two displayed inequalities yield 
\eqref{eq:leqSqrt2} .
\end{proof}

Now we prove the main result of this section, which extends Proposition \ref{prop:inequality}:

\begin{prop}\label{prop:leqSqrt2}
For any spiral $\gamma$, $E(\gamma)\leq 2$ with equality only if $|\gamma|\equiv\sqrt{2}$.
\end{prop}
\begin{proof}
  Lemma \ref{lem:leqSqrt2} together with Lemma \ref{lem:continuous}  yields $E(\gamma)\leq 2$ via a polygonal approximation. Next suppose that $E(\gamma)=2$. Let $r$, $R$ be the initial and final heights of $\gamma$. If $r=R$, then  $2=E(\gamma)=\mathcal E(|\gamma|)$ which yields $|\gamma|\equiv\sqrt2$. Suppose towards a contradiction that $r<R$.
  Let $P_i$, $i=1$, $2$, $\dots$ be a sequence of polygonal spirals converging to $\gamma$, with initial and final heights $r_i$, $R_i$. We may assume for convenience that $r\leq r_i<R_i\leq R$. Let $w_i$ be the weight functions for $P_i$ given by Lemma \ref{lem:leqSqrt2}. Set $\ol w_i:=w_i/L(P_i)$ on $[r_i,R_i]$ and $\ol w_i:=0$ elsewhere. Then $\int_{r}^{R} \ol w_i(h)dh=1$.
  By Lemma  \ref{lem:continuous}, for any given $\epsilon>0$, we may choose $i$ so large that
$
E(P_i)\geq 2-\epsilon.
$
Then by Lemma \ref{lem:leqSqrt2},
$$
2-\epsilon\leq E(P_i)\leq \int_{r}^{R} \ol w_i(h)\mathcal{E}(h)  dh\leq \sup_{[r,R]}\mathcal{E}\leq 2.
$$
So $\sup_{[r,R]}\mathcal{E}=2$. 
Since $\mathcal{E}=2$ only at $\sqrt{2}$,  $[r,R]\ni\sqrt2$. So the set of heights  $h\in[r,R]$ with $\mathcal{E}(h)\geq 2-\sqrt{\epsilon}$ forms a subinterval $[x_\epsilon^-, x_\epsilon^+]$. It follows that
 $$
2-\epsilon\leq\int_{r}^{R}  \ol w_i(h)\mathcal{E}(h)dh 
\le 
   -\sqrt{\epsilon}\left(\int_{r}^{x_\epsilon^-}  \ol w_i(h)dh+\int_{x_\epsilon^+}^{R}  \ol w_i(h)dh  \right)+ 2.
$$
So
$
\int_{r}^{x_\epsilon^-}  \ol w_i(h)dh+\int_{x_\epsilon^+}^{R} \ol w_i(h)dh \le \sqrt\epsilon.
$
But $\ol w_i\geq 1/\big(L(P_i)\sqrt{(R/r)^2-1}\,\big)$. Thus
$$
R-r\leq\sqrt\epsilon L(P_i)\sqrt{(R/r)^2-1}+x_\epsilon^+-x_\epsilon^-.
$$
Letting $\epsilon\to 0$, we obtain $r=R$, since
$
x_\epsilon^\pm \to\sqrt{2},
 $
 and $L(P_i)$ is bounded above.
Hence we arrive at the desired contradiction. 
\end{proof}

\section{Proof of Theorem \ref{thm:main}}\label{sec:proof}

By Proposition \ref{prop:minimizer} there exists a minimal inspection curve $\gamma\colon[a,b]\to\R^3$, which we may assume to have constant speed. As we described in  Section \ref{sec:integral}, to establish  \eqref{eq:main}  it suffices to show that $E(\gamma)\leq 2$. By Proposition \ref{prop:unfolding}, $E(\gamma)=E(\tilde\gamma)$ where  $\tilde\gamma$ is the unfolding of $\gamma$. By Proposition \ref{prop:decomposition}, $\tilde\gamma$ admits a spiral decomposition, generated by a collection of mutually disjoint open sets $U_i\subset[a,b]$, $i\in I$. 
Set $U_0:=[a,b]\setminus \cup_i\ol U_i$, and let $\tilde\gamma_i:=\tilde\gamma|_{\ol U_i}$,  $\tilde\gamma_0:=\tilde\gamma|_{U_0}$. Then
\be\label{eq:E-gamma}
E(\tilde\gamma)=\frac{H(\tilde\gamma)}{L(\tilde\gamma)}=\frac{1}{L(\tilde\gamma)}\sum_iH(\tilde\gamma_i)=
\frac{1}{L(\tilde\gamma)}\left(L(\tilde\gamma_0)E(\tilde\gamma_0)+ \sum_iL(\tilde\gamma_i)E(\tilde\gamma_i)\right),
\ee
where we define $L(\tilde\gamma_0):=\int_{U_0}|\tilde\gamma_0'(t)|dt$. So
$L(\tilde\gamma_0)+ \sum_iL(\tilde\gamma_i)=L(\tilde\gamma)$.
If $L(\tilde\gamma_0)=0$, then we may disregard the first term in the summation above. Otherwise,
by definition of spiral decomposition, $\tilde\alpha(t)=\pi/2$ for almost all $t\in U_0$. Thus, by Proposition \ref{prop:IE},
\be\label{eq:E-gamma-0}
E(\tilde\gamma_0)=\frac{1}{\mu(U_0)}\int_{U_0} \mathcal{E}\left(|\tilde\gamma(t)|,\frac{\pi}{2}\right)dt=\frac{1}{\mu(U_0)}\int_{U_0}\mathcal{E}\big(|\tilde\gamma(t)|\big)dt\leq 2.
\ee
Furthermore, by Proposition  \ref{prop:inequality} or \ref{prop:leqSqrt2},
\be\label{eq:E-gamma-1}
E(\tilde\gamma_i)\leq 2,
\ee
assuming $U_i\neq\emptyset$.
So it follows that $E(\tilde\gamma)\leq 2$, as desired.
It remains to characterize the case of equality in \eqref{eq:main}, which corresponds to $E(\gamma)= 2$. Then $E(\tilde\gamma)=2$, which yields that the terms $E(\tilde\gamma_i)$ and $E(\tilde\gamma_0)$ in \eqref{eq:E-gamma} must all be equal to $2$. But the inequality in \eqref{eq:E-gamma-1} must be strict by Proposition \ref{prop:leqSqrt2}, since $\tilde\gamma_i$ are strict spirals by definition of spiral decomposition. So $\tilde\gamma$ cannot contain any strict spirals or $U_i=\emptyset$, which means that $U_0=[a,b]$ or $\tilde\gamma$ has constant height. Furthermore, equality in \eqref{eq:E-gamma-0} implies that 
$\mathcal{E}(|\tilde\gamma(t)|)\equiv2$ which can happen only when $|\tilde\gamma(t)|\equiv\sqrt2$. So we conclude that   $\gamma$ has constant height $\sqrt{2}$, since unfoldings preserve height. 
Now let $\ol\gamma:=\gamma/\sqrt2$ be the projection of $\gamma$ into $\S^2$. Then $L(\ol\gamma)=L(\gamma)/\sqrt2=4\pi/\sqrt2$. 
Recall that, since $\gamma$ is an inspection curve, the horizon circles generated by points of $\gamma$ cover $\S^2$. Since $|\gamma|\equiv\sqrt{2}$, these circles have (spherical) radius $\pi/4$ and are centered at points of $\ol\gamma$. 
Thus  $\ol\gamma$ satisfies the hypothesis of the following proposition, which will complete the proof of Theorem \ref{thm:main}.

\begin{prop}\label{prop:baseball}
Let $\gamma\colon[a,b]\to\S^2$ be a closed  constant speed curve with $L(\gamma)=4\pi/\sqrt2$. Suppose that the 
distance between any point of $\S^2$ and $\gamma$ is at most $\pi/4$. Then $\gamma$ is a simple $\C^{1,1}$ curve which traces consecutively $4$ semicircles of  length $\pi/\sqrt2$.
\end{prop}

It remains then to establish the above proposition. To this end we need:

\begin{lem}[Crofton-Blaschke-Santalo \cite{santalo:1942}]\label{lem:crofton}
Let $\gamma\colon[a,b]\to\S^2$ be a rectifiable curve, and
for every point $p\in\S^2$, and $0\leq\rho\leq\pi/2$, let $C_\rho(p)\subset\S^2$ denote the circle of  radius $\rho$ centered at $p$. Then
$
L(\gamma)=\frac{1}{4\sin(\rho)}\int_{p\in\S^2} \#\gamma^{-1}\big(C_\rho(p)\big)\,dp.
$
\end{lem}

\noindent
For the rest of this section we assume that $\gamma$ satisfies the hypothesis of the last proposition.
Then, since $L(\gamma)=4\pi/\sqrt2$, applying the last lemma with $\rho=\pi/4$ to  $\gamma$  yields
$$
\underset{p\in\S^2}{\textup{Ave}}\;\#\gamma^{-1}\big(C_{\frac\pi4}(p)\big)
=\frac{1}{4\pi}\int_{p\in\S^2} \#\gamma^{-1}\big(C_{\frac\pi4}(p)\big)dp
=\frac{1}{4\pi}L(\gamma)\,4\,\sin\left(\frac\pi4\right)
=2.
$$
Furthermore, note that $C_{\frac\pi4}(p)$ must intersect $\gamma$ for all $p\in\S^2$, since the distance of $p$ from $\gamma$ cannot be bigger than $\pi/4$ by assumption. So, since $\gamma$ is closed, $\#\gamma^{-1}(C_{\frac\pi4}(p))\geq 2$ for almost all $p\in\S^2$. Now since the average of $\#\gamma^{-1}(C_{\frac\pi4}(p))$ is  $2$, it follows that

\begin{lem}\label{lem:2-points}
For almost every $p\in\S^2$,  
$
 \#\gamma^{-1}\big(C_{\frac\pi4}(p)\big)=2.
$
\end{lem}

By a \emph{side} of a circle $C$ in $\S^2$ we mean either of  the two closed disks in $\S^2$ bounded by $C$. If the radius of $C$ is less than $\pi/2$, then the disk with radius less than $\pi/2$ is called the \emph{inside} of $C$ and the other disk is called the \emph{outside} of $C$. By \emph{strictly inside} or \emph{strictly outside} we mean the interior of inside and interior of outside respectively.

\begin{lem}\label{lem:outside-length}
For any point $p\in\S^2$, the portion of $\gamma$ which lies outside $C_{\frac\pi4}(p)$ has length at least $\pi$.
\end{lem}
\begin{proof}
By assumption, $\gamma$ intersects $C_{\frac\pi4}(-p)$, which has distance $\pi/2$ from $C_{\frac\pi4}(p)$. Furthermore, since $\gamma$ is closed, there must exist at least two segments of $\gamma$ which connect $C_{\frac\pi4}(-p)$ and $C_{\frac\pi4}(p)$.
\end{proof}

For the rest of this section, we will assume that $\gamma$ is reparameterized so that $[a,b]=[0,2\pi]$, and identify $[0,2\pi]$ with the unit circle $\S^1\simeq \R/(2\pi \mathbf{Z})$. Furthermore we fix an orientation on $\S^1$. Then for every pair of distinct points $t$, $s\in\S^1$, we let $[t,s]$ denote the segment in $\S^1$ whose orientation from $t$ to $s$ agrees with the orientation of $\S^1$. 

\begin{lem}\label{lem:tangent-cone}
For every $t\in\S^1$, the tangent cone $T_t\gamma$  is a line.
\end{lem}
\begin{proof}
Let $s_i\in\S^1$ be a sequence of points converging to $t$ from the left hand side (with respect to the orientation of $\S^1$). Since $\gamma$ has non-vanishing speed, it cannot be locally constant. Thus we may assume, after passing to a subsequence, that $\gamma(s_i)\neq\gamma(t)$. Then the secant rays $\ell_i$ in $\R^3$ which emanate from $\gamma(t)$ and pass through $\gamma(s_i)$ are well-defined. Let $\ell$ be a limit of $\ell_i$. Similarly, we can consider the secant rays $\ell_i'$ generated by points $s_i'\in\S^1$ converging to $t$ from the right hand side, and let $\ell'$ be a limit of $\ell_i'$. We claim that the angle between $\ell$ and $\ell'$ is $\pi$. Suppose not. Then there exist points $s$, $s'\in\S^1$ arbitrary close to $t$ and with $(s,s')\ni t$ such that the angle between the geodesic segments $\gamma(t)\gamma(s)$ and $\gamma(t)\gamma(s')$ in $\S^2$  is less than $\pi$. Consequently, there exists an open set $S$ of circles  of radius $\pi/4$ in $\S^2$ such that for every $C\in S$ we have $\gamma(s)$, $\gamma(s')$ lie strictly inside $C$  while $\gamma(t)$ lies strictly outside $C$.  Thus, by Lemma \ref{lem:2-points}, there exists a circle $C\in S$ which intersects $\gamma$ in only two points. So the portion of $\gamma$ which lies outside $C$ is a subset of $\gamma([s,s'])$. But since $s$ and $s'$ may be chosen arbitrarily close to $t$, the length of $\gamma([s,s'])$ may be arbitrarily small. Hence we obtain the desired contradiction via Lemma \ref{lem:outside-length}. So the angle between $\ell$ and $\ell'$ is $\pi$ as claimed. Now since $\ell$ and $\ell'$ where arbitrary limits of the right and left secant rays of $\gamma$ at $t$, and all these limits are tangent to $\S^2$, it follows that $\ell$ and $\ell'$ are unique. Hence $T_t\gamma=\ell\cup\ell'$ as desired. 
\end{proof}

Now for each $t\in\S^1$, the left and right unit tangent vectors of $\gamma$, 
$
u^\pm_\gamma(t)
$
are well-defined with $u^+_\gamma(t)=-u^-_\gamma(t)$.

\begin{lem}\label{lem:gamma-tangent}
Let $C\subset \S^2$ be a circle of radius $\pi/4$.
Suppose that there exists an interval $[t,s]\subset\S^1$  such that $\gamma(t)$ lies on $C$ while $\gamma((t,s])$ lies strictly inside $C$. Then $\gamma$ is transversal to $C$ at $\gamma(t)$.
\end{lem}
\begin{proof}
Suppose towards a contradiction that  $\gamma$ is tangent to $C$ at $\gamma(t)$.  Let $C'$ be a circle of radius $\pi/4$ in $\S^2$ which passes through $\gamma(t)$ and  is transversal to $C$ at $\gamma(t)$ with $u^+_\gamma(t)$ pointing outside $C'$. Then there exist $r\in (t,s)$ such that $\gamma(r)$ lies strictly outside $C'$. Furthermore, choosing $C'$ sufficiently close to $C$, we can ensure  that $\gamma(s)$ lies strictly inside $C'$. Next, by perturbing the center of $C'$, we may find  another circle $C''$ of radius $\pi/4$ such that $\gamma(t)$ and $\gamma(s)$ lie strictly inside $C''$ while $\gamma(r)$ lies strictly outside $C''$. Since $C''$ may be chosen freely from an open set of circles in $\S^2$, we may assume by Lemma \ref{lem:2-points} that $C''$ intersects $\gamma$ at only two points. Thus the portion of $\gamma$ lying outside $C''$ is a subset of $\gamma([t,s])$. But $\gamma([t,s])$  can  have arbitrarily small length, since we may choose $s$ as close to $t$ as desired. Thus we obtain a contradiction by Lemma \ref{lem:outside-length}.
\end{proof}

We say that a circle $C\subset\S^2$ \emph{supports} $\gamma$ at a point $p$ of $\gamma$ provided that $C$ passes through $p$ and $\gamma$ lies on one side of $C$. Furthermore, if the radius of $C$ is less than $\pi/2$, then we assume that $\gamma$ lies outside $C$.

\begin{lem}\label{lem:support-circles}
Through each point of $\gamma$ there pass a pair of support circles of radius $\pi/4$ which  lie outside each other.
\end{lem}
\begin{proof}
Let $C$ be one of the two circles of radius $\pi/4$  tangent to $\gamma$ at $\gamma(t)$. Suppose there is a point $t'\in\S^1$ such that $\gamma(t')$ lies strictly inside $C$. Let $D$ be the disk of radius $\pi/4$ bounded by $C$, and $I$ be the closure of the  component of $\gamma^{-1}(\inte(D))$ which contains $t'$. By Lemma \ref{lem:2-points}, $\gamma$ cannot lie entirely in $C$. Thus $I$ is a proper interval in $\S^1$. By Lemma \ref{lem:gamma-tangent}, $\gamma$ is transversal to $C$ at the end points of $I$. In particular there are points $s_1$, $s_2\in\S^1$ close to each of the end points of $I$ such that $\gamma(s_i)$ lie strictly outside $C$, and $t$, $s_1$, $t'$, $s_2$ are arranged cyclically in $\S^1$.
Perturbing the center of $C$, we may find a circle $C'$ of radius $\pi/4$ such that $\gamma(t)$ and $\gamma(t')$ lie strictly inside $C'$, while $\gamma(s_i)$  lie strictly outside $C'$. It follows that $C'$ intersects $\gamma$ at least $4$ times, which contradicts Lemma \ref{lem:2-points}, since $C'$ may be chosen freely from an open set of circles in $\S^2$.
\end{proof}

In the terminology of \cite{ghomi-howard2014}, the conclusion of Lemma \ref{lem:support-circles} means that $\gamma$ has \emph{double positive support}. Using this lemma we next show:

\begin{lem}\label{lem:simple}
$\gamma$ is simple.
\end{lem}
\begin{proof}
Suppose  that there are distinct points $t$, $s\in \S^1$ with $\gamma(t)=\gamma(s)$. Let $r_1$, $r_2$ be points of $\S^1$ which lie in the interior of different segments of $\S^1$ determined by $s$ and $t$, so that  $s$, $r_1$, $t$, $r_2$ are cyclically arranged in $\S^1$. 
By Lemma \ref{lem:support-circles}, there exists  a circle $C$ of radius $\pi/4$ in $\S^2$ which supports $\gamma$ at $\gamma(t)=\gamma(s)$. Furthermore, by Lemma \ref{lem:2-points}, $\gamma$ cannot lie completely on $C$. So we may choose $r_i$ so that at least one of the points
$\gamma(r_1)$, $\gamma(r_2)$ lies strictly outside $C$. Then we may translate $C$ to obtain a circle $C'$ of the same radius such that $\gamma(t)=\gamma(s)$ lies strictly inside $C'$ while $\gamma(r_i)$ lie strictly outside $C'$. Hence $C'\cap\gamma$ consist of at least $4$ points. Furthermore  $C'$ may be chosen from an open set of circles of radius $\pi/4$ in $\S^2$. Thus we obtain a violation of Lemma \ref{lem:2-points}.
\end{proof}

For a planar curve,  double positive support is equivalent to  \emph{positive reach} introduced by Federer \cite{federer1959}. Thus the last two lemmas imply that $\gamma$ has positive reach.

\begin{lem}\label{lem:C11}
$\gamma$ is $\C^{1,1}$.
\end{lem}
\begin{proof}
Since $\gamma$ has finite length, there exists a point in $\S^2\setminus\gamma$, which we may assume to be $(0,0,1)$ after a rotation. Let $\pi\colon\S^2\setminus\{(0,0,1)\}\to\R^2$ be the stereographic projection, and set $\tilde\gamma:=\pi\circ\gamma$. Since $\pi$ preserve circles, and by Lemma \ref{lem:support-circles} $\gamma$ has 
double positive support, $\tilde\gamma$ has double positive support as well.  Furthermore, by Lemma \ref{lem:simple} $\gamma$ has two sides in $\S^2$. The support circles of $\gamma$ must lie in opposite sides of $\gamma$ at each point; otherwise the tangent cone would be a ray (or $\gamma$ would have a cusp) which is not possible by Lemma \ref{lem:tangent-cone}. Thus the support circles of $\tilde\gamma$ must lie on the opposite sides of $\tilde\gamma$ as well.  Consequently $\tilde\gamma$ is $\C^{1,1}$ by \cite[Thm. 1.2]{ghomi-howard2014}; see also \cite[prop. 1.4]{lytchak}. 
\end{proof}

Next we observe that:

\begin{lem}\label{lem:antipodal}
Let $C$ be a support circle of $\gamma$ of radius $\pi/4$. Then $C\cap \gamma$ is either a pair of antipodal points of $C$ or  else is a semicircle of $C$. 
\end{lem}
\begin{proof}
We claim that (i) every closed semicircle of $C$ intersects $\gamma$, and (ii) every open semicircle of $C$ intersects $\gamma$ in a connected set. These properties easily imply that  $\gamma\cap C$ is either a pair of antipodal points of $C$, a closed semicircle of $C$, or the entire $C$. The last possibility is not allowed, because by Lemma \ref{lem:simple} $\gamma$ is simple; therefore, if $\gamma$ covers $C$, it must coincide with $C$, which would violate Lemma \ref{lem:2-points}. So it remains to establish the claims.
To see (i) suppose that there exists a closed semicircle of $C$ which does not intersect $\gamma$. Then
moving the center of $C$ by a small distance towards the center of  that semicircle yields
a circle $C'$ of radius $\pi/4$ disjoint from $\gamma$. Obviously all circles of radius $\pi/4$ which are close to $C'$ will be disjoint from $\gamma$ as well, which would violate Lemma \ref{lem:2-points}. 
To see (ii) suppose that there exists an open semicircle $S$ of $C$ which intersects $\gamma$ in a disconnected set.
Then there exist  points 
$t_1$, $t_2$, $s\in\S^1$, with $s\in(t_1,t_2)$ such that $\gamma(t_i)\in S$  while $\gamma(s)$ lies strictly outside $C$. Either  $\gamma((t_2,t_1))$ lies entirely on $C$ or not. In the former case there exist a point $s'\in(t_2, t_1)$ such that 
$\gamma(s')$ lies in the open semicircle of $C$ which is disjoint from $S$; in the latter case
there exists a point $s'\in(t_2, t_1)$ such  that $\gamma(s')$ lies strictly outside $C$. In either case,
moving the center of $C$ by a small distance towards the midpoint of $S$ will yield a circle $C'$ of radius $\pi/4$ such that
$\gamma(t_i)$ lie strictly inside  $C'$ while $\gamma(s)$, $\gamma(s')$ lie strictly outside $C'$. But $t_1$, $s$, $t_2$, $s'$ are cyclically arranged in $\S^1$. So perturbing the center of $C'$ yields an open set of circles of radius $\pi/4$ each of which intersects $\gamma$ at least $4$ times, which again contradicts Lemma \ref{lem:2-points}. 
\end{proof}

The last lemma  leads to the proof of Proposition \ref{prop:baseball} via the notion of \emph{nested partitions of a circle} employed in \cite{ghomi:rosenberg}, see also \cite{umehara2,thorbergsson&umehara}. By Lemma \ref{lem:simple}, $\gamma$ bounds a topological disc $D\subset\S^2$. By Lemmas \ref{lem:support-circles} and \ref{lem:C11} through each point $p\in\gamma$ there passes a circle $C_p$ of  radius $\pi/4$ which lies in $D$. It follows from Lemma \ref{lem:C11} that $C_p$ is unique. Thus if we set
 $
 [p]:=\gamma\cap C_p, \;\text{and}\; P:=\{[p]\mid p\in\gamma\},
 $
then $P$ will be a partition of $\gamma$. A partition $P$ of a topological circle is  \emph{nested} provided that no element of $P$ separates the components of any other element, i.e.,  for every $[p]\in P$ and $q\in\gamma\setminus [p]$, $[q]$ lies in a connected component of $\gamma\setminus [p]$. 
 
 \begin{lem}
 The partition $P$ of $\gamma$ is nested. 
 \end{lem}
 \begin{proof}
 If not there are distinct support circles $C$, $C'$ of $\gamma$ of radius $\pi/4$ contained in $D$ such that
 $C$ has points  in different components of $\gamma\setminus C'$. In particular neither $C\cap\gamma$ nor $C'\cap\gamma$ is connected. So by Lemma \ref{lem:antipodal}, $C$ and $C'$ intersect $\gamma$ in  two points each, say $C\cap\gamma=\{p,q\}$ and $C'\cap\gamma=\{p',q'\}$. Each of the segments  $pq$ and $qp$ of $C$ separate $D$ into two components. Thus each of the segments $p'q'$ and $q'p'$ of $C'$  must intersect each of the segments $pq$ and $qp$. Furthermore, each of these intersections must occur in the interior of the segments, because the interior of each segment is disjoint from $\gamma$. Thus $C$ and $C'$ intersect at least $4$ times. So  $C=C'$, which is the desired contradiction. 
  \end{proof}
  
  Finally we  invoke the following  fact established in \cite[Lem. 2.2]{ghomi:rosenberg}. A partition is \emph{nontrivial} provided that it contains more than one element.
  
  \begin{lem}
  Any nontrivial nested partition of a topological circle  contains at least two elements which are connected subsets of the circle.
 \end{lem}
 
So there are  two distinct elements $[p_1]$, $[p_2]\in P$ such that $C_{p_i}\cap\gamma$ is a connected set. Consequently, by Lemma \ref{lem:antipodal}, $C_{p_i}\cap\gamma$ are semicircles.  Thus $\gamma$ contains a pair of disjoint semicircles which curve toward $D$. Similarly,  by repeating the above argument for the other domain $D'$ in $\S^2$ bounded by $\gamma$, we obtain
 two disjoint semicircles in $\gamma$ which curve toward  $D'$.  Since each
semicircle has length $\pi/\sqrt{2}=L(\gamma)/4$, the semicircles cover $\gamma$,  which completes the proof of Proposition \ref{prop:baseball}.

\section*{Acknowledgments}
We thank Joseph O'Rourke for posting the  sphere inspection problem on MathOverflow \cite{orourkeMO} which prompted this work, and commenters on MathOverflow, including Pietro Majer, Anton Petrunin, Jean-Marc Schlenker, and Gjergji Zaimi for their  observations.  Thanks also to Steven Finch for sending us a translation of  Zalgaller's article \cite{finch2019translation}.

\addtocontents{toc}{\protect\setcounter{tocdepth}{0}}

\addtocontents{toc}{\protect\setcounter{tocdepth}{1}}
\bibliography{references}

\end{document}